\documentclass{amsart}
\usepackage{amsfonts}
\usepackage{amssymb}
\usepackage{ dsfont }
\usepackage{amsmath}
\usepackage{tikz}
\numberwithin{equation}{section}
\usepackage{graphicx}
\usepackage{ mathrsfs }
\newcommand{\eqdef}{\stackrel{\scriptscriptstyle\rm def}{=}}

\begin{document}

\title[Markovian random products]{Stability of the  Markov operator \\
and synchronization\\
of  Markovian random products}

\author[L. J.~D\'\i az]{Lorenzo J. D\'\i az}
\address{Departamento de Matem\'atica PUC-Rio, Marqu\^es de S\~ao Vicente 225, G\'avea, Rio de Janeiro 22451-900, Brazil}
\email{lodiaz@mat.puc-rio.br}

\author[E. Matias]{Edgar Matias}
\address{Instituto de Matem\'atica Universidade Federal do Rio de Janeiro, Av. Athos da Silveira Ramos 149, Cidade Universit\'aria - Ilha do Fund\~ao, Rio de Janeiro 21945-909,  Brazil}
\email{edgarmatias9271@gmail.com}

\begin{abstract}
	We study Markovian random products on a large class of
	``$m$-dimensional'' connected compact metric spaces (including products of closed intervals and trees). We introduce a splitting condition, generalizing the classical one by Dubins and Freedman, 
	 and prove that this condition implies
	the
	 asymptotic stability of the corresponding Markov operator and
	 (exponentially fast) synchronization. \end{abstract}

\begin{thanks}{This paper was partially part of the PhD thesis of EM (PUC-Rio) supported by CAPES.
The authors  thank the hospitality of Centro de Matem\'atica of Univ. of Porto (Portugal) and
IMPAN (Warsaw, Poland)  for their hospitality and 
partial support of 
EU Marie-Curie IRSES ``Brazilian-European partnership in Dynamical
Systems" (FP7-PEOPLE-2012-IRSES 318999 BREUDS). 
LJD is partially supported by CNPq and CNE-Faperj (Brazil).
The authors warmly thank  D. Malicet, K. Gelfert, and A. Zdunik for their useful comments.
}	
\end{thanks}

\keywords{Asymptotic stability, iterated function system, Markov operators,
stationary measures, Markovian random products, synchronization}
\subjclass[2000]{%
37B25, 
37B35, 60J05,  47B80.}

\maketitle
\newtheorem{mtheorem}{Theorem}
\newtheorem{theorem}{\textbf{Theorem}}[section]
\newtheorem{acknowledgement}{Acknowledgement}
\newtheorem{case}{\textbf{Case}}
\newtheorem{claim}[theorem]{\textbf{Claim}}
\newtheorem{conclusion}{\textbf{Conclusion}}
\newtheorem{condition}{\textbf{Condition}}
\newtheorem{corollary}{\textbf{Corollary}}
\newtheorem{criterion}{\textbf{Criterion}}
\newtheorem{definition}[theorem]{\textbf{Definition}}
\newtheorem{example}{\textbf{Example}}
\newtheorem{lemma}[theorem]{\textbf{Lemma}}
\newtheorem{proposition}[theorem]{\textbf{Proposition}}
\newtheorem{remark}[theorem]{\textbf{Remark}}
\newtheorem{solution}{\textbf{Solution}}
\newtheorem{summary}{\textbf{Summary}}
\newtheorem{notation}{\textbf{Notation}}
\newtheorem{question}{\textbf{Question}}

\section{Introduction}
The study of random products of maps has a long history and for independent and 
identically distributed (i.i.d.) random products the first fundamental results go back to Furstenberg \cite{Fur}. In this paper, we study Markovian  random products.
 Given a compact metric space $M$, finitely many continuous maps $f_{i}\colon M \to M$, $i=1,\dots, k$, and a Markov shift $(\Sigma_{k},\mathscr{F},\mathbb{P},\sigma)$ 
(see the precise definition below) the map $\varphi\colon \mathbb{N}\times \Sigma_{k}\times M\to M$ defined by 
\begin{equation} 
\label{e.randomproductmap}
\left\lbrace
\begin{array}{l}
\varphi(n,\omega,x)\eqdef f_{\omega_{n-1}}\circ\dots \circ f_{\omega_{0}}(x)\eqdef f_{\omega}^{n}(x)\quad \text{for}\quad n\geq 1,\\
\varphi(0,\omega,x)\eqdef x
\end{array}\right.
\end{equation}
is called a \emph{random product of maps} over the Markov shift
(or \emph{a Markovian product of random maps}).

The study of random products is of major importance in the study of the asymptotic behavior of Markov chains. This is for example supported by the fact that every homogenous Markov chain admits a certain representation by an i.i.d. random product, see Kifer \cite{Kifer}.   For additional references on Markovian random products see \cite{Recurrent, Edalat, Virtser,Gui,Royer}, and references therein.

Associated to the Markovian random product $\varphi$ there is a family of homogeneous Markov chains
\[
	Z^{x}_{n}(\omega)\eqdef(\omega_{n-1},X_{n}^{x}(\omega)),
	\quad n\geq 1 \text{ and }x\in M,
\]
	 where $X_{n}^{x}$ is the 
random variable $\omega\mapsto f_{\omega}^{n}(x)$.
We study the asymptotic behavior of this family (synchronization) and also the dynamics 
of the corresponding Markov operator.

We will focus on certain ``$m$-dimensional'' compact spaces, which include
products of compact intervals and trees. More precisely, we will assume that the
ambient space is of the form $Y^m=Y \times \cdots \times Y$, where $Y$ is a metric space such that
 every connected non-singleton subset of $Y$ has nonempty interior.
In the particular case of the interval (i.e., $Y=[0,1]$ and $m=1$), the study of products of random maps goes back to Dubins and Freedman \cite{Dubins}, where i.i.d. random products of monotone maps are considered and the asymptotic stability of 
  the corresponding Markov operator is obtained. There they
  assumed a certain splitting condition on the system (for details see Definition~\ref{def1} and the discussion in Remark~\ref{r.discussion}) which we  also invoke. 

In the above context, we will prove that there is a unique stationary measure and, in fact, show the asymptotic stability of the Markov operator (which implies the former), see Theorem~\ref{stable}. 
Let us briefly discuss the results which are our main motivation (see, 
in chronological ordering, \cite{Dubins, Hu, BhLe, Recurrent, Edalat}).
 As observed above, in \cite{Dubins}, assuming a splitting condition, it is proved the asymptotic stability of the Markov operator associated to an i.i.d. random product of injective maps of the closed interval.
 This result was generalized to higher dimensional non-decreasing monotone maps in \cite{BhLe} .
 The asymptotic stability of the Markov operator is studied in \cite{Recurrent} for Markovian random products of maps\footnote{In 
 \cite{Recurrent,Edalat} it is used the terminology {\emph{recurrent I.F.S.}} for a Markovian random products of maps.}  which are contracting in average (this generalizes the results in \cite{Hu} for contractions). More precisely, they show that there exists a unique stationary measure which attracts every probability measure from a  certain ``representative" subspace. Finally, in \cite{Edalat} it is  introduced the notion of
 weak hyperbolicity (i.e., no contraction-like assumptions are involved, see the discussion below) and proved that for every weakly hyperbolic Markovian random product there is a unique stationary measure. With a slight abuse of terminology, we can refer to the settings in \cite{Recurrent,Edalat} as hyperbolic-like ones and observe that our setting is ``genuinely non-hyperbolic".
Here we consider the Markovian case instead of the i.i.d. case and our approach allows to  consider higher dimensional compact metric spaces  and continuous maps which may not be injective, in particular extending the results in \cite{BhLe} to a larger class of random products of maps,
see Section~\ref{sss.monotone} .

%

As noted above, we are also interested in
the asymptotic behavior of orbits and study  synchronization.
This phenomenon  was first observed  by 
 Huygens \cite{Huygens} in the synchronizing movement of two pendulum clocks hanging from a wall and since then has been investigated in several areas, see \cite{syncrhonizationbook}.
 A random product $\varphi$ is said to be \emph{synchronizing} if
 random orbits of different initial points converge to each other with probability $1$
 (see Section \ref{sync} for the precise definition).  
In this context, a first result was 
 obtained by Furstenberg in \cite{Fur} for i.i.d. products of random matrices on projective spaces.
 The occurrence of synchronization for  
i.i.d. random products of circle homeomorphisms is now completely characterized by Antonov's theorem
 \cite {Antonov,KlepAntonov}
and its generalization in
 \cite {Malicet}.
We prove that a Markovian random product defined on 
 a connected compact subset of $\mathbb{R}^m$
  satisfying the splitting condition is (exponentially fast) synchronizing, see
  Theorem~\ref{t.mcsincronizacao}. In the i.i.d.
  case and when $m=1$, this result is obtained in \cite{Malicet}. In Section~\ref{ss.satisfying}  we present some classes of 
  higher dimensional
  random products satisfying the splitting condition, see also Theorem~\ref{specialclass}.
We note that Theorem~\ref{t.mcsincronizacao} is a consequence of a general result of ``contraction of volume" in
compact spaces, see Theorem~\ref{presincronizacao}.


  Finally, let us observe that,
  besides its intrinsic interest, the study of the dynamics of Markov operators associated to random products is closely related to the study of the corresponding step skew products. The latter is currently a subject of intensive research and we refrain to select a list of references. We just would like to
 mention  \cite{Volk}, closely related to our paper, where it is proved that generically Markovian random products of 
 $C^1$ diffeomorphisms on the interval $[0,1]$ (with image strictly contained in $[0,1]$) has a finite number 
 of stationary measures and to each one it corresponds a physical measure of the associated skew product.

  This paper is organized as follows. In Section \ref{s.main} we state the main definitions and the precise 
  statements of our results. In Section~\ref{s.newconsequences}
  we state consequences of the splitting hypothesis for general compact spaces. 
  In Section~\ref{s.consequences} we draw consequences from the splitting condition in one-dimensional settings.
 In Section~\ref{s.Markovop} we prove the asymptotic stability of the Markov operator (Theorem~\ref{stable}) and 
 a strong version of the Ergodic Theorem (Corollary \ref{ergodictheorem}).  In Section \ref{s.sincro} we study the synchronization of Markovian random systems (Theorem~\ref{t.mcsincronizacao}) which will be a consequence of a contraction result for measures for general compact metric spaces (Theorem~\ref{presincronizacao}). Finally, in Section~\ref{ss.splittingcondition} we prove 
 Theorem~\ref{specialclass} which states sufficient conditions for the splitting property.

\section{Main results}
\label{s.main}

\subsection{General setting} 
Let $(E,\mathscr{E})$ be a measurable space and consider a transition probability 
 $P\colon E\times \mathscr{E}\to [0,1]$, i.e.,  
 for every $x\in E$ the mapping $A\mapsto P(x,A)$ is a probability measure on $E$ and 
 for every $A\in \mathscr{E}$ the mapping $x \mapsto P(x,A)$ is measurable with respect to $\mathscr{E}$. Recall that a 
 measure $m$ on $E$ is called a \emph{stationary measure} with respect to $P$ if 
 $$
 m(A)=\int P(x,A)\, dm(x),
 \quad
 \mbox{for every $A\in \mathscr{E}$.}
 $$

Suppose that $E\eqdef\{1,\dots,k\}$ is a finite set endowed with the discrete sigma-algebra $\mathscr{E}$.
 Consider the space of unilateral sequences $\Sigma_{k}=E^{\mathbb{N}}$ endowed 
with the product sigma-algebra $\mathscr{F}=\mathscr{E}^{\mathbb{N}}$. Given a transition probability $P$ on $E$ and  stationary measure $m$, there is a unique probability measure $\mathbb{P}$ 
on $\Sigma_{k}$ such that the sequence of coordinate mappings on $\Sigma_{k}$
is a homogeneous Markov chain with probability transition 
$P$ and starting measure $m$.  For details see, e.g.,
\cite[Chapter 1]{Revuz}.
 The measure $\mathbb{P}$ is the \emph{Markov measure} of the pair $(P,m)$.
 Denote by $\sigma$ the shift map on $\Sigma_{k}$. The measure $\mathbb{P}$ is $\sigma$-invariant  and the metric dynamical system $(\Sigma_{k},\mathscr{F},\mathbb{P},\sigma)$ is called a \emph{Markov shift}.

Throughout this paper $(M,d)$ is a compact metric space. Let $\varphi(n,\omega,x)=f^{n}_{\omega}(x)$ be a random product on $M$ over the Markov shift $(\Sigma_{k},\mathscr{F},\mathbb{P},\sigma)$ as in~\eqref{e.randomproductmap}. Fix $x\in M$ and note that
the sequence of random variables $X_{n}^{x}\colon \omega \mapsto f_{\omega}^{n}(x)$, in general, is not a Markov chain on the probability space $(\Sigma_{k},\mathscr{F},\mathbb{P})$. However, it turns out that the sequence 
$Z_{n}^{x}(\omega)=(\omega_{n-1}, X_{n}^{x}(\omega))$, $n\geq 1$, is 
a Markov chain with range on the space $\widehat M\eqdef E\times M$ (see \cite{Recurrent}) with probability transition given by 
\begin{equation}\label{transition}
\widehat P((i,z),B)\eqdef P(i, \{j\colon (j,f_{j}(z))\in B \})
\end{equation}
for every $B\in \mathscr{E}\otimes \mathscr{B}$, where $P$ is the transition probability on $E$ associated to $\mathbb{P}$.

Let $\mathcal{M}_{1}(\widehat M)$ be the space of probabilities on $\widehat M$ equipped with the weak-$*$ topology.
The {\emph{Markov operator}} associated to $\varphi$ is given by 
\begin{equation}\label{handyform}
 T\colon \mathcal{M}_{1}(\widehat M)\to \mathcal{M}_{1}(\widehat M),\qquad  T\widehat\mu(B)\eqdef \int \widehat P((i,z),B)\,d\widehat\mu(i,z).
 \end{equation}
 This operator is called  \emph{asymptotically stable} if there is a probability measure $\widehat \mu$ such that $T\widehat \mu=\widehat \mu$
and for every $\widehat \nu\in \mathcal{M}_{1}(\widehat M)$ it holds 
$$
\lim_{n\to \infty} T^{n}\widehat \nu= \widehat \mu,
$$
in the weak-$*$ topology. Note that a measure $\widehat \mu$ that satisfies $T\widehat \mu=\widehat\mu$ is, by definition, a stationary measure. 

In this paper ,we obtain sufficient conditions for the asymptotic stability of the Markov operator. We also 
investigate the (common) asymptotic behavior of the family $((X_{n}^{x})_{n\in \mathbb{N}})_{x\in X}$.

 Finally, note
that a probability transition and a stationary measure on the finite set $E=\{1,\dots,k\}$ are given respectively
by a \emph{transition matrix} and a \emph{stationary probability vector}. 
 Recall that
a $k\times k$ matrix  $P=(p_{ij})$ is a 
\emph{transition matrix} if 
$p_{ij}\ge 0$ for all $i,j$ and 
for every $i$ it holds
$\sum_{j=1}^{k}p_{ij}=1$.
A \emph{stationary probability vector} associated to $P$  is a 
  vector $\bar p=(p_{1},\ldots,p_{k})$ whose elements are 
  non-negative real numbers, sum up to $1$, and satisfies $\bar p \, P= \bar p$.  

A Markov shift $(\Sigma_{k},\mathscr{F},\mathbb{P},\sigma)$ 
with transition matrix $P=(p_{ij})$
is called \emph{primitive} if there is $n\geq 0$ such that 
all entries of $P^{n}$ are positive. 
It is called \emph{irreducible} if  
for every $\ell,r\in \{1,\dots,k\}$ there is $n=n(\ell,r)$ such that
$P^{n}=(p^n_{ij})$ satisfies $p^n_{\ell,r}>0$.
An irreducible transition matrix has a unique positive stationary
probability vector $\bar p=(p_{i})$, see 
\cite[page 100]{Kemeny}. Clearly, every primitive Markov shift is irreducible. Finally, recall that a sequence $(a_{1},\dots,a_{\ell})\in \{1,\dots,k\}^{\ell}$ is called \emph{admissible} with respect to $P=(p_{ij})$ if $p_{a_{i}a_{i+1}}>0$ for every $i=1,\dots, \ell-1$.

\subsection{The  splitting condition and the stability of the Markov operator}
Before introducing the main result of this section, let us define the splitting condition.

Let $(Y,d_{0})$ be a separable metric space and let $M$ be a compact subset of $Y^{m}$, $m\geq 1$. 
In $Y^{m}$ consider the metric $d(x,y)\eqdef \sum_{i}d_{0}(x_{i},y_{i})$ and
the continuous projections $\pi_{s}\colon 
X^{m}\to X$, given by $\pi_{s}(x)\eqdef x_{s}$, where $x=(x_{1},\dots ,x_{m})$.

\begin{definition}[Splitting condition]\label{def1}
\emph{
Let $\varphi(n,\omega,x)=f^{n}_{\omega}(x)$ be a random product on a compact subset $M\subset Y^m$ over the Markov shift $(\Sigma_{k},\mathscr{F},\mathbb{P},\sigma)$ and let $P=(p_{ij})$ its transition matrix.
We say that $\varphi$ \emph{splits} if
there exist admissible sequences $(a_{1},\dots,a_{\ell})$ and $(b_{1},\dots b_{r})$ with $a_{\ell}=b_{r}$, such that $M_{1}\eqdef f_{a_{\ell}}\circ\dots\circ f_{a_{1}}(M)$ and $M_{2}\eqdef f_{b_{r}}\circ\dots\circ f_{b_{1}}(M)$
satisfy\begin{equation}\label{splitgene}
\pi_{s}(f_{\omega}^{n}(M_{1}))\cap \pi_{s}(f_{\omega}^{n}(M_{2}))=\emptyset
\end{equation}
 for every $\omega\in \Sigma_k$, every $n\geq 0$, and  
every projection $\pi_{s}$, $i=1, \dots,m$.}
\end{definition}

The splitting equation \eqref{splitgene} means that the action of the random product on the sets $M_{1}$ and $M_{2}$ remains  disjoint when projected in all ``directions''.
We will present in Section~\ref{ss.satisfying}
some
relevant classes of random products for which if the splitting 
equation \eqref{splitgene} holds for $n=0$ then the splitting condition is satisfied
(i.e., equation \eqref{splitgene} holds for every $n\ge 0$). 
This 
is the case for instance, when $m=1$ and all maps of the random product are injective.

%
%

\begin{remark}\label{r.discussion}
\emph{The
splitting condition was introduced in \cite[Section 5]{Dubins},
 in the context of i.i.d. random products of
monotone maps of the interval. Our definition is more general and coincides with the original one in the case of monotone maps of the interval.
The above definition is somewhat similar to the strong open set condition (SOSC) which is very often 
studied in the context of iterated function systems, see \cite{Hu}, although mainly with a deterministic focus.
 The SOSC is for instance satisfied if the interiors of the images $f_i(M)$ are all pairwise disjoint.} 
\end{remark}
%
%
%
%

Our first result deals with the dynamics of the Markov operator.

\begin{mtheorem}[Asymptotic stability of the Markov operator]
\label{stable}Let $Y$ be a separable metric space such that every 
(non-singleton) connected subset of $Y$ has nonempty interior.  Let $M\subset Y^{m}$ be a connected compact subset and consider a random product $\varphi$ on $M$ over a primitive Markov shift and suppose that $\varphi$ splits.
  Then
 the associated Markov operator is asymptotically stable.
\end{mtheorem}

 
A natural context where the above theorem applies is when $Y$ is the real line.

As a consequence of Theorem \ref{stable} we obtain a strong version 
of the Ergodic Theorem. Under the assumptions of Theorem \ref{stable}, 
the Markov chain $ Z_{n}^{x}(\omega)=(\omega_{n-1},X_{n}^{x}(\omega))$
has a unique stationary measure $\widehat\mu$. This allows  us to apply the Breimans' ergodic theorem \cite{Breiman} to get  a subset $\Omega_{x}\subset \Sigma_{k}$ of full measure (depending on $x$)
such that 
 \begin{equation}\label{ergodic}
 \lim_{n\to \infty}\frac{1}{n}\sum_{i=0}^{n-1}\widehat \phi(Z_{j}^{x}(\omega))=\int \widehat \phi(i,x)\,d\widehat\mu(i,x),
 \end{equation}
 for every $\omega\in \Omega_{x}$ and every continuous function $\widehat \phi\colon  \widehat M\to \mathbb{R}$. 

Thus, we can describe the time average of the sequence $X_{n}^{x}$ for 
every $x$. For every $\omega\in \Omega_{x}$ and every continuous function $\phi\colon M\to \mathbb{R}$ we have 
that 
\begin{equation}\label{ergodic3}
\lim_{n\to \infty}\frac{1}{n}\sum_{i=0}^{n-1}\phi(f_{\omega}^{j}(x))=\sum _{i=1}^{k}\int \phi\,d\mu_{i},
\end{equation}
where $\mu_{i}$ is the $i$-section of $\widehat\mu$ (see  Section~\ref{ss.Markovop} for precise definitions).
To give a more detailed description of the right hand side in \eqref{ergodic3} let us introduce two 
definitions.

Consider a transition matrix $P=(p_{ij})$ and a positive stationary probability vector
$\bar p=(p_{1},\dots,p_{k})$
of $P$. The \emph{inverse transition matrix} associated to $(P,\bar p)$ is the matrix $Q=(q_{ij})$ where 
\begin{equation}\label{e.Q}
 q_{ij}=\frac{p_{j}}{p_{i}}p_{ji}.
\end{equation}
 Note that $Q$ is a transition matrix and $\bar p$ is a stationary probability vector of $Q$. 
 The Markov measure associated to $(Q,\bar{p})$ is called the \emph{inverse Markov measure} and 
 is denoted by $\mathbb{P}^{-}$.

We say that a measurable map $\pi\colon\Sigma_{k}\to  M$ is an \emph{invariant map} 
 of the random product if 
 \begin{equation}\label{f1omega}
f_{\omega_{0}}( \pi(\sigma(\omega)))= \pi(\omega), \quad \mbox{$\mathbb{P}^{-}$-almost everywhere.}
 \end{equation}
 Invariant maps sometimes give relevant information about the 
random system. For instance, in the theory of contracting iterated function systems, see \cite{Hu},
the coding map is the unique (a.e.) invariant map and it is essential in the description 
of properties of attractors and stationary measures of i.i.d random products.

Recall that given a measurable map $\phi \colon X \to Z$ and a measure $\mu$ in $X$
the {\emph{pushforward of $\mu$ by $\phi$}}, denoted $\phi_* \mu$, is the measure on $Z$
defined by  $\phi_{*} \mu (A)\eqdef  \mu (\phi^{-1} (A))$.

\begin{corollary}[Strong version 
of the Ergodic Theorem]
\label{ergodictheorem}
Let $Y$ be a separable metric space such that every 
(non-singleton) connected subset of $Y$ has nonempty interior.  Let $M\subset Y^{m}$ be a connected compact subset and consider a random product $\varphi(n,\omega,x)=f^{n}_{\omega}(x)$ on $M$ over a primitive Markov shift
$(\Sigma_{k},\mathscr{F},\mathbb{P}, \sigma)$ and suppose that $\varphi$ splits.
Then there is a unique ($\mathbb{P}^-$-a.e.) invariant map
 $\pi\colon \Sigma_{k}\to M$. Moreover,  for every $x\in M$ and 
 for $\mathbb{P}$-almost every $\omega\in\Sigma_{k}$ it holds
$$
\lim_{n\to \infty}\frac{1}{n}\sum_{i=0}^{n-1}\phi(f_{\omega}^{i}(x))=\int \phi\,d(\pi_{*}\mathbb{P}^{-}),
$$
for every continuous function $\phi\colon M\to \mathbb{R}$.
\end{corollary}

\begin{remark}
\label{r.coding}{\em{
The proof of the corollary provides the unique invariant map: $\pi$ is the so-called
{\emph{(generalised) coding map}} (see \eqref{coding}) that is defined on the subset of weakly hyperbolic sequences of $\Sigma_k$ (see \eqref{weakly}).}}
\end{remark}

\subsection{Synchronization} \label{sync}
We now consider random products over irreducible Markov shifts.
Let $\varphi(n,\omega,x)=f^{n}_{\omega}(x)$ be a random product
over a Markov shift $(\Sigma_{k},\mathscr{F},\mathbb{P}, \sigma)$. We say that $\varphi$ is \emph{synchronizing} if 
for every pair $x$ and $y$ we have that 
$$
\lim_{n\to \infty}d(f_{\omega}^{n}(x), f_{\omega}^{n}(y))=0, \quad
\mbox{for $\mathbb{P}$-a.e. $\omega\in \Sigma_{k}$.}
$$

Let us start with a general result that states a weak form of synchronization in general compact metric spaces.

\begin{mtheorem}[Contraction of measures]\label{presincronizacao}
Let $\varphi(n,\omega,x)=f^{n}_{\omega}(x)$ be a random product on a compact metric space $M\subset Y^{m}$ over an irreducible Markov
shift $(\Sigma_{k},\mathscr{F},\mathbb{P}, \sigma)$. Suppose that $\varphi$ splits. 
Let $P=(p_{ij})$ be the corresponding transition matrix 
and suppose that there is $u$ such that $p_{uj}>0$ for every $j$. 
Then there is $q<1$  such that for every probability measure $\mu$ on $Y$ and for every $s=1,\dots, m$, for $\mathbb{P}$-almost every $\omega$ there is  $C_{s}(\omega)>0$ such that
$$
\mu ( \,\pi_{s}(f^{n}_{\omega}(M)))\leq C_{s}(\omega) q^{n}, \quad \mbox{for every} \quad n\geq 1.
$$
\end{mtheorem}

The next result states the synchronization of Markovian random products on compact 
subsets of $\mathbb{R}^{m}$ in a strong version: uniform and exponential. We consider in $\mathbb{R}^{m}$ the metric $d(x,y)\eqdef \sum_{i}|x_{i}-y_{i}| $.

\begin{mtheorem}[Synchronization]\label{t.mcsincronizacao}
Let $\varphi(n,\omega,x)=f^{n}_{\omega}(x)$ be a random product on a compact set $M\subset \mathbb{R}^{m}$ over an irreducible Markov shift $(\Sigma_{k},\mathscr{F},\mathbb{P}, \sigma)$. Suppose that $\varphi$ splits. 
Let $P=(p_{ij})$ be the corresponding transition matrix 
and suppose that there is $u$ such that $p_{uj}>0$ for every $j$. Then there is $q<1$ such that for $\mathbb{P}$-almost every $\omega$ there is  $C(\omega)$ such that
$$
\mathrm{diam} \,( f^{n}_{\omega}(M)) \leq C(\omega) q^{n}, \quad \mbox{for every} \quad n\geq 1.
$$
\end{mtheorem}

The above results extends  
\cite[Corollary 2.11]{Malicet} stated for  i.i.d. random products of monotone
(injective) interval maps.
Note that our result holds in 
higher dimensions.

\subsection{Random products with the splitting property}
\label{ss.satisfying}
We now describe some classes
of random products for which  the splitting 
equation \eqref{splitgene} for $n=0$ guarantees the splitting condition.

\subsubsection{Injective maps on ``one-dimensional''  compact spaces}
Let $Y$  be a compact metric space
 such that every 
(non-singleton) connected subset of $Y$ has nonempty interior. 
Let $\varphi(n,\omega,x)=f_{\omega}^{n}(x)$ be  a Markovian random product  on $Y$ such that the maps $f_{i}$ are injective.  If there exist admissible sequences $(a_{1},\dots,a_{\ell})$ and $(b_{1},\dots b_{r})$ with $a_{\ell}=b_{r}$, such that 
$$
f_{a_{\ell}}\circ\dots\circ f_{a_{1}}(Y)\cap f_{b_{r}}\circ\dots\circ f_{b_{1}}(Y)=\emptyset
$$
then $\varphi$ splits. This is  a direct consequence of the injectivity hypothesis and the fact that there is only one direction to project the space.

A natural context where the above comment applies is when $Y$ is an interval.
 Observe that in the one-dimensional setting, there are certain topological restrictions. For example,
although every (non-singleton) connected subset of $S^{1}$ has nonempty interior,  
continuous injective maps 
  on the circle have to be homeomorphisms and hence the splitting equation 
 \eqref{splitgene} cannot be satisfied. 
   We refer to \cite{Malicet,Anna} for recent results in the study of i.i.d random products of homeomorphisms on the circle.
Let us observe, however, that there are other  `one-dimensional''  metric spaces which are not intervals which carry Markovian random products 
that split, for example trees.

\subsubsection{Monotone maps on compact subsets of $\mathbb{R}^{m}$}
\label{sss.monotone}
 
Let $f\colon \mathbb{R}^{m}\to \mathbb{R}^{m}$ be a continuous injective map. Let $f^{i}\colon \mathbb{R}^{m}\to \mathbb{R}$ be the 
coordinates function of $f$, i.e., $\pi_i \circ f$. We write $f=(f^{1},\dots,f^{m})$. Since $f$ is injective then for every $i$ and every fixed $x_{1},\dots,x_{j-1},x_{j+1},\dots x_{m}$ the map $x\mapsto f_{i}(x_{1},x_{j-1},x,x_{j+1},\dots, x_{m})$
is monotone. Following \cite{BhLe} (although our setting is more general) we will introduce  a special class of injective maps, called \emph{monotone maps.}
For that  we need to  
introduce two definitions. 

A monotone map $g\colon \mathbb{R} \to \mathbb{R}$ is of  \emph {type $+$}
if it is increasing  and is of \emph{type $-$} if it is decreasing.  We say that $f^{i}$ is of \emph{type $(t_{1},\dots, t_{m})\in \{+,-\}^{m}$} if  
 for every $j$ and every $(x_{1},\dots,x_{j-1},x_{j+1},\dots x_{m})\in \mathbb{R}^{m-1}$ the map 
$$
x \mapsto f^{i}(x_{1},x_{j-1},x,x_{j+1},\dots, x_{m})
$$
is of type $t_{j}$.

Given $(t_{1},\dots, t_{m})\in  \{+,-\}^{m}$  a map 
 $f=(f^{1},\dots,f^{m})$ belongs to  $\mathcal{S}(t_{1},\dots, t_{m})$ 
 if and only if
\begin{itemize}
\item[(1)] If $t_{j}=t_{1}$ then $f^{j}$ is of type $(t_{1},\dots, t_{m})$,

\item[(2)] If $t_{j}\neq t_{1}$ then $f^{j}$ is of type
 $(s_{1},\dots, s_{m})$, where $s_{\ell}=+$ if $t_{\ell}=- $ and $s_{\ell}=-$ if $t_{\ell}=+$.
\end{itemize}

We denote by $\mathcal{S}_M(t_{1},\dots, t_{m})$ the maps $f \colon M\to M$ 
that admit an extension to a map $\widehat f \colon \mathbb{R}^m \to \mathbb{R}^m$ 
in $\mathcal{S}(t_{1},\dots, t_{m})$. 
 We say that a  Markovian random product $\varphi(n,\omega,x)=f^{n}_{\omega}(x)$ 
 is in
$ \widetilde{\mathcal{S}}_M(t_{1},\dots, t_{m})$ if
the maps $f_i$ belongs to
$\mathcal{S}_M(t_{1},\dots, t_{m})$ for every $i$.

 Let $A$ and $B$ be bounded subsets of $\mathbb{R}$. If $\sup A<\inf B$ then we write $A<B$. In particular, if $A<B$ then  
$A \cap B=\emptyset$. 

\begin{mtheorem}\label{specialclass}
 Given 
 $(t_{1},\dots, t_{m}) \in   \{+,-\}^{m} $ 
consider  a Markovian random product $\varphi(n,\omega,x)=f^{n}_{\omega}(x)$
in $\widetilde{\mathcal{S}}_M(t_{1},\dots, t_{m})$ defined on a compact set  $M\subset R^{m}$. 
Suppose that there exist admissible sequences $(a_{1},\dots,a_{\ell})$ and $(b_{1},\dots b_{r})$ with $a_{\ell}=b_{r}$, such that 
$M_{1}\eqdef f_{a_{\ell}}\circ\dots\circ f_{a_{1}}(M)$ and $M_{2}\eqdef f_{b_{r}}\circ\dots\circ f_{b_{1}}(M)$
 satisfy
\begin{itemize}
\item
$\pi_{s}(M_{1})< \pi_{s}(M_{2})$ if $t_{s}=+$ and 
\item
$\pi_{s}(M_{1})> \pi_{s}(M_{2})$ if $t_{s}=-$
\end{itemize}
for every projection $\pi_{s}$. Then $\varphi$ splits.
\end{mtheorem}

In particular, under the hypotheses of the previous theorem and adequate assumptions on the transition matrix
(see Theorems \ref{stable} and \ref{t.mcsincronizacao}) it follows the exponential synchronization
and the asymptotic stability of the Markov operator. The case of
i.i.d. random products of maps in  $\widetilde{\mathcal{S}}_M(1,\dots, 1)$
 was treated in \cite{BhLe}, where the asymptotic stability of the Markov operator is obtained. Here 
 we consider the Markovian random  maps in the set $\widetilde{\mathcal{S}}_M$ defined by
$$
\widetilde{\mathcal{S}}_M \eqdef \bigcup_{(t_1,\ldots,t_m)\in \{+,-\}^m} \,
\widetilde{\mathcal{S}}_M(t_{1},\dots, t_{m}).
$$

\subsubsection{Minimal iterated function systems}
Given a compact subset $M\subset \mathbb{R}^{m}$ and finitely many continuous maps $f_{i}\colon M \to M$, $i=1,\dots, k$, we consider its associated  \emph{iterated function system} denoted by $\textrm{IFS}(f_{1},\dots,f_{k})$. Let $E_n=\{1,\dots,k\}^n$. The IFS is called 
\emph{minimal} if for every $x\in M$ it holds

$$
\mathrm{closure} \left( \bigcup_{n\in \mathbb{N}} \, \, \bigcup_{(a_1,\dots,a_n)\in E_n} f_{a_{n}}\circ\dots \circ f_{a_{1}}(x) \right)=M .
$$

The next result is a consequence of Theorem~\ref{specialclass}.
\begin{corollary}\label{c.corollaryefect}
Consider a compact set $M\subset \mathbb{R}^{m}$ with non-empty interior and
a random product  $\varphi(n,\omega,x)=f_{\omega}^{n}(x)$
over 
a Markov shift $(\Sigma_{k},\mathscr{F},\mathbb{P}, \sigma)$
 in $\widetilde{\mathcal{S}}_{M}$. Assume  $\mathbb{P}$ has full support, the $\mathrm{IFS}(f_{1},\dots,f_{k})$ is minimal, and there is a sequence $\omega\in\Sigma_{k}$ such that 
\begin{equation}\label{weaklyly}
\bigcap f_{\omega_{0}}\circ \dots \circ f_{\omega_{n}}(M) =\{p\}
\end{equation}
for some point $p$. Then $\varphi$ splits. 
\end{corollary}

Condition \eqref{weaklyly} is guaranteed, for instance,  if 
 some composition $f_{a_{\ell}}\circ \dots \circ f_{a_{1}}$  is a contraction. 
Note that this condition  is compatible with the minimality of the IFS. 
Indeed, 
Corollary~\ref{c.corollaryefect} is mainly illustrative, since the minimality condition is harder to check than the splitting condition (which in many cases can be easily obtained).

\subsubsection{Injective maps in a box}
Let $Y$  be a compact metric space
 such that every 
(non-singleton) connected subset of $Y$ has non-empty interior. 
Let $\varphi(n,\omega,x)=f_{\omega}^{n}(x)$ be a Markovian random product defined on $Y$. Suppose that there is a compact subset (a box) $J\subset Y$ such that $f_{i}(J)\subset J$
and $f_{i|J}$ is injective
for every $i=1,\dots,k$.  If there exist admissible sequences $(a_{1},\dots,a_{\ell})$ and $(b_{1},\dots b_{r})$ with $a_{\ell}=b_{r}$, such that 
$$
\begin{array}{l}
f_{a_{\ell}}\circ\dots\circ f_{a_{1}}(Y)\cap f_{b_{r}}\circ\dots\circ f_{b_{1}}(Y)=\emptyset \quad 
\mbox{and}\\
f_{a_{\ell}}\circ\dots \circ f_{a_{1}}(Y)\cup f_{b_{r}}\circ\dots\circ f_{b_{1}}(Y)\subset J,
\end{array}
$$
 then $\varphi$ splits.

\section{Consequences of the splitting hypothesis}
\label{s.newconsequences}
We now explore consequences of the splitting condition for 
Markovian random products 
$\varphi(n,\omega,x)=f^{n}_{\omega}(x)$ on a general compact metric space $M\subset Y^{m}$ over an
irreducible
 Markov
shift $(\Sigma_{k},\mathscr{F},\mathbb{P}, \sigma)$.  Note that with these assumptions
$\mathbb{P}^-$ is well defined.

We begin with some general definitions.
Consider
the {\emph{cylinders}}
 $$
 [a_{0}\dots a_{\ell}]\eqdef \{\omega\in \Sigma_k \colon \omega_{0}=a_{0},\dots,\omega_{\ell}=a_{\ell}\}
 \subset\Sigma_k,
 $$ 
which is a semi-algebra that generates the Borel sigma-algebra of $\Sigma_k$.
Given a Markov measure $\mathbb{P}$ on $\Sigma_{k}$ 
with
 transition matrix $P=(p_{ij})$ and a stationary measure
$\bar p=(p_{1},\dots,p_{k})$ we have that 
 $$
 \mathbb{P}([a_{0}\ldots a_\ell])=
 p_{a_{0}}p_{a_{0}a_{1}}\ldots p_{a_{\ell-1}a_\ell}.
 $$ 
 A cylinder $C=[a_{0}a_{1}\dots a_{\ell}]$ is $\mathbb{P}$-\emph{admissible} if $\mathbb{P}(C)>0$.

For the random product $\varphi(n,\omega,x)=f_{\omega}^{n}(x)$ and the projection $\pi_{s}$ 
we define two family of subsets of $\Sigma_{k}$. First, for each $x\in \pi_{s}(M)$ and $n\geq 1$ 
let
\begin{equation}\label{e.snx} 
S_{n}^{x}(s)\eqdef \{\omega\in \Sigma_{k}\colon x\in\pi_{s}( f_{\omega_{0}}\circ\cdots\circ
f_{\omega_{n-1}}(M))\} 
\end{equation}
and observe that
\begin{equation}
\label{e.inclusionSn}
S_{n+1}^{x}(s) \subset S_n^{x}(s).
\end{equation}
Second, 
for each cylinder $C$ of size $N$ define
\begin{equation}\label{particao}
\Sigma_{n}^C\eqdef 
\{\omega\in \Sigma_{k} \colon \sigma^{iN}(\omega)
 \cap C=\emptyset\,\, \mbox{for all}\,\, i=0,\ldots ,n-1\}
  \end{equation}
 and note that
 \begin{equation}
\label{e.inclusioncylinders}
\Sigma_{n+1}^{C} \subset \Sigma_n^C.
\end{equation}
 
\begin{proposition} \label{p.l.menorouigual}
Let $M\subset Y^{m}$ be a compact metric space and 
$\varphi(n,\omega,x)=f^{n}_{\omega}(x)$ be
a random product on $M$ over an
irreducible  Markov shift $(\Sigma_{k},\mathscr{F},\mathbb{P},\sigma)$.
Let $[\xi_{0}\dots \xi_{N-1}]$
 and $[\eta_{0}\dots\eta_{N-1}]$ be
$\mathbb{P}^{-}$-admissible cylinders such that $\xi_{0}=\eta_{0}$, $\xi_{N-1}=\eta_{N-1}$, and
\begin{equation}\label{injective}
\pi_{s}(f_{\omega}^{n}( f_{\xi_{0}}\circ\dots\circ f_{\xi_{N-1}}(M)))\cap \pi_{s}(f_{\omega}^{n}(f_{\eta_{0}}\circ\dots\circ f_{\eta_{N-1}}(M)))=\emptyset
\end{equation}
 for every $\pi_{s}$, every $n\geq 0$, and every $\omega$.
If
\[
0<\mathbb{P}^{-}([\xi_{0}\ldots\xi_{N-1}])
\leq
 \mathbb{P}^{-}([\eta_{0}\ldots\eta_{N-1}]),
\]
then the cylinder $W\eqdef [\xi_{0}\ldots\xi_{N-1}]$ satisfies
$$
\mathbb{P}^{-}(W)>0 
\quad
\mbox{and}
\quad
 \mathbb{P}^{-}(S_{\ell N}^x(s))\leq \mathbb{P}^{-}(\Sigma_{\ell}^W),
 $$ 
  for every $\ell\geq 1$, every $s$ and every $x$.
\end{proposition}

\begin{proof}
By the  admissibility of
 $[\xi_{0}\dots \xi_{N-1}]$ we have
\begin{equation}\label{e.forinstance}
0<\mathbb{P}^{-}([\xi_{0}\ldots\xi_{N-1}])
\leq
 \mathbb{P}^{-}([\eta_{0}\ldots\eta_{N-1}]).
 \end{equation}  

Fix a projection $\pi_{s}$. Take $x\in \pi_{s}(M)$ and  for each $\ell\geq 1$ 
define the following two families of cylinders:
 \begin{equation}\label{e.sigmaxj}
 \Sigma_{x}^{\ell}(s)\eqdef\{[a_{0}\ldots a_{\ell N-1}]\subset
  \Sigma_{k}\colon x\in \pi_{s}(f_{a_{0}}\circ\cdots\circ
f_{ a_{\ell N-1}}(M))\}
 \end{equation}
 and 
 \begin{equation}
 \label{e.Eell}
 E^{\ell}\eqdef\{[a_{0}\ldots a_{\ell N-1}]\subset
  \Sigma_{k}\colon  \sigma^{iN}([a_{0}\ldots a_{\ell N-1}])
  \cap W=\emptyset,\, \forall \,i=0,\ldots ,\ell-1\} .
  \end{equation}
Note that
 \begin{equation}
 \label{e.twoidentities}
 S_{\ell N}^{x}(s)=
  \bigcup_{C
 \in\Sigma_{x}^{\ell}(s)} C \quad \mbox{and}\quad \Sigma_{\ell}^W=\bigcup_{C
 \in E_{\ell}} C.
\end{equation}

 For each $\ell\geq 1$ we now  introduce a ``substitution function" 
 $F_{\ell}: \Sigma_{x}^{\ell}(s)\to E^{\ell}$.
First, for each cylinder $C=[\alpha_{0}\dots \alpha_{\ell N-1}]\in \Sigma_{x}^{\ell}(s)$ 
 we consider its sub-cylinders
 $[\alpha_0\dots \alpha_{N-1}]$, $[\alpha_N\dots \alpha_{2N-1}], \dots, 
 [\alpha_{(\ell-1)N}\dots \alpha_{\ell N-1}]$
 and use the   concatenation notation
 $$
  [\alpha_0\dots \alpha_{\ell N-1}]\eqdef
   [\alpha_0\dots \alpha_{N-1}]\ast [\alpha_N\dots \alpha_{2N-1}]\ast
 \cdots  \ast [\alpha_{(\ell-1)N}\dots \alpha_{\ell N-1}].
 $$
 In a compact way, we write
 $$
 C=C_0 \ast C_1 \ast \cdots \ast C_{\ell-1},
 \quad C_i\eqdef  [\alpha_{iN}\dots \alpha_{(i+1)N-1}].
 $$
  With this notation we define $F_\ell$ by
 $$
 F_\ell(C)\eqdef
 F_\ell(C_0 \ast C_1 \ast \cdots \ast C_{\ell-1})=
 C_0' \ast C_1' \ast \cdots \ast C_{\ell-1}',
 $$
 where $C_i'=C_i$ if $C_i\ne [\xi_0\dots \xi_{N-1}]$ and 
 $C_i'=[\eta_0\dots \eta_{N-1}]$ otherwise.
 Note that by definition  $ F_\ell(C)\subset E^\ell$ for each $C\in\Sigma_x^\ell(s)$
 and hence the map is well defined.

\begin{claim}\label{c.substitutionmap1}
For every
$C\in \Sigma_{x}^{\ell}(s)$ it holds
$\mathbb{P}^{-}(C)\leq \mathbb{P}^{-} (F_{\ell}(C))$.  
\end{claim}  

\begin{proof}
Recalling that
 $\eta_{0}=\xi_{0}$ and $\eta_{N-1}=\xi_{N-1}$,
 from equation~\eqref{e.forinstance}
 we immediately get the following:
 For every  $r, j\ge 0$ and  every pair of cylinders $[a_{0}\ldots a_{j}]$ and 
 $[b_{0}\ldots b_{r}]$ it holds
\begin{enumerate} 
\item 
 $\mathbb{P}^{-}([a_{0}\ldots a_{j}\xi_{0}\ldots\xi_{N-1}
 b_{0}\ldots b_{r}])\leq 
 \mathbb{P}^{-}([a_{0}\ldots a_{j}\eta_{0}\ldots\eta_{N-1}
 b_{0}\ldots b_{r}])$,
 \item$
  \mathbb{P}^{-}([\xi_{0}\ldots\xi_{N-1}
 b_{0}\ldots b_{r}])\leq 
 \mathbb{P}^{-}([\eta_{0}\ldots\eta_{N-1}
 b_{0}\ldots b_{r}])$, and
 \item
  $
  \mathbb{P}^{-}([a_{0}\ldots a_{j}\xi_{0}\ldots\xi_{N-1}]\leq 
 \mathbb{P}^{-}([a_{0}\ldots a_{j}\eta_{0}\ldots\eta_{N-1}].
  $
  \end{enumerate}
The inequality $\mathbb{P}^{-}(C)\leq \mathbb{P}^{-} (F_{\ell}(C))$ now follows from the definition of
$F_\ell$.
 \end{proof}

\begin{lemma}\label{l.c.substitutionmap}
The map $F_{\ell}$ is injective.
\end{lemma}  

\begin{proof}
Using the concatenation notation above, consider cylinders 
  $ C=C_0 \ast C_1 \ast \cdots \ast C_{\ell-1}$ 
  and $\widetilde C =\widetilde C_0 \ast \widetilde C_1 \ast \cdots  \ast \widetilde C_{\ell-1}$
  with $C\ne \widetilde C$ in $\Sigma_{x}^{\ell}(s)$. Write
  $$
 F_\ell(C)=C_0' \ast C_1' \ast \cdots \ast C_{\ell-1}'
 \quad
 \mbox{and}
 \quad
 F_\ell(\widetilde C )=\widetilde C_0' \ast \widetilde C_1' \ast \cdots  \ast \widetilde C_{\ell-1}'.
$$

 Assume, by contradiction, that
$F_{\ell}(C)=F_{\ell}(\widetilde{C})$ and hence
$C_i'=\widetilde C_i'$ for all $i=0,\dots, N-1$. 
Since $C\ne \widetilde C$
there is a first  $i$ such that $C_i\ne \widetilde C_i$. 
Then, by the definition of $F_\ell$, either $C_i= [\xi_0\dots \xi_{N-1}]$ and 
$\widetilde C_i= [\eta_0\dots \eta_{N-1}]$  or vice-versa. Let us assume that the
first case occurs.

If  $i=0$ then 
the definition of $\Sigma_{x}^{\ell}(s)$ 
implies that
$$
x\in \pi_{s}(f_{\xi_0}\circ \cdots \circ f_{\xi_{N-1}} (M)) \cap
\pi_{s}(f_{\eta_0}\circ \cdots \circ f_{\eta_{N-1}} (M)),
$$
contradicting the hypothesis in  \eqref{injective}. Thus we can assume that $i>0$.
Write $(i-1)N-1=r$ and consider the cylinders
\[
\begin{split}
[\gamma_{0}\ldots \gamma_{r}] & \eqdef C_0 \ast C_1 \ast \cdots \ast C_{i-1}=
\widetilde C_0 \ast \widetilde C_1 \ast \cdots  \ast \widetilde C_{i-1},\\
[\gamma_{r+N}\ldots \gamma_{\ell N-1}] & \eqdef C_{i+1} \ast C_1 \ast \cdots \ast C_{\ell-1},\\
[\widetilde \gamma_{r+N}\ldots \widetilde \gamma_{\ell N-1}] & \eqdef
\widetilde C_{i+1} \ast C_1 \ast \cdots \ast \widetilde C_{\ell-1}
\end{split}
\]
and the corresponding finite sequences
\[
\begin{split}
\gamma_0 \cdots \gamma_{\ell N-1} & \eqdef
 {\gamma_0}  \dots {\gamma_{r}} \, {\xi_0} \cdots {\xi_{N-1}}\,
{\gamma_{r+N}}  \cdots \gamma_{\ell N-1},\\
\widetilde \gamma_0 \cdots  \widetilde \gamma_{\ell N-1} & \eqdef
 {\gamma_0}  \dots {\gamma_{r}} \, {\eta_0} \cdots {\eta_{N-1}}\,
\widetilde \gamma_{r+N}  \cdots \widetilde \gamma_{\ell N-1}.
\end{split}
\]
Since  $f_{i}(M)\subset M$ we have
\[
\begin{split}
 f_{\gamma_0} \circ \cdots \circ f_{\gamma_{\ell N-1}}
(M)
&\subset
 f_{\gamma_0} \circ \cdots \circ f_{\gamma_{r}}
\circ f_{\xi_0}\circ \cdots \circ f_{\xi_{N-1}} (M),\\
 f_{\widetilde \gamma_0} \circ \cdots \circ f_{\widetilde \gamma_{\ell N-1}}
(M)
&\subset
 f_{\gamma_0} \circ \cdots \circ f_{\gamma_{r}}
\circ f_{\eta_0}\circ \cdots \circ f_{\eta_{N-1}} (M).
\end{split}
\]
Hence,
by the definition of $\Sigma_x^\ell(s)$ in \eqref{e.sigmaxj}, we have 
\[
\label{e.contradiction18abril}
x\in \pi_{s}(f_{\gamma_0} \circ \cdots \circ f_{\gamma_{r}}
\circ f_{\xi_0}\circ \cdots \circ f_{\xi_{N-1}} (M)) \cap \pi_{s}(f_{\gamma_0} \circ \cdots \circ f_{\gamma_{r}}
\circ f_{\eta_0}\circ \cdots \circ f_{\eta_{N-1}} (M)),
\]
contradicting \eqref {injective}.
Thus $C=\widetilde C$, proving the lemma.
\end{proof}

Condition $\mathbb{P}^{-}(W)>0$ follows from the choice of $W$.
To prove that
$\mathbb{P}^{-}(S_{\ell N}^x(s))\leq\mathbb{P}^{-}(\Sigma_{\ell }^W)$ note that
 \[
 \begin{split}
 \mathbb{P}^{-}(S_{\ell N}^x(s))
 &\underset{(\textrm{a})}{=}\sum_{C\in\Sigma^{\ell}_{x}(s)} \mathbb{P}^{-}(C)
  \underset{(\textrm{b})}{\leq}
 \sum_{C\in\Sigma^{\ell}_{x}(s)} \mathbb{P}^{-}(F_{\ell}(C))
\\
   &
  \underset{(\textrm{c})}{=}\mathbb{P}^{-}\big( \bigcup_{C\in\Sigma^{\ell}_{x}(s)} F_{\ell}(C)
  \big) \underset{(\textrm{d})}{\le}  \mathbb{P}^{-} (\Sigma_{\ell }^W),
  \end{split}
  \] 
  where (a) follows from the disjointedness of the cylinders $C\in \Sigma_x^j(s)$,
  (b)
   from  Claim~\ref{c.substitutionmap1}, 
   (c)
 from the injectivity of $F_{j}$ (Lemma~\ref{l.c.substitutionmap}), and 
  (d) from  $F_{j}(C)\in E_j \subset Q^{j}$.
The proof of  the proposition is now complete. 
\end{proof}

\section{Consequences of the splitting hypothesis in $m$-dimensional spaces}
\label{s.consequences}
Throughout this section $\varphi$ denotes a  random product  on a compact  and connected
metric space $M\subset Y^{m}$ over a Markov shift $(\Sigma_{k},\mathscr{F},\mathbb{P},\sigma)$. To state the main theorem of this section (Theorem~\ref{abundance}) we need
some definitions.

For each $\xi\in \Sigma_{k}$ we consider its \emph{fibre} defined by 
$$
I_{\xi}\eqdef \bigcap_{n\geq 0}f_{\xi_{0}}\circ \dots\circ f_{\xi_{n}}(M).
$$
Every fibre is a non-empty compact set: note that  $(f_{\xi_{0}}\circ \dots\circ f_{\xi_{n}}(M))_{n\in \mathbb{N}}$ is a 
sequence of nested compact sets. 
Also note that 
$I_\xi$ is a connected set.

The subset $S_{\varphi}\subset \Sigma_{k}$ of \emph{weakly hyperbolic sequences} is defined by 
\begin{equation}\label{weakly}
S_{\varphi}\eqdef \{\xi \in \Sigma_{k}\colon I_{\xi}\,\, \text{is a singleton}\}.
\end{equation}

The main result in this section is the following:

\begin{theorem}\label{abundance} 
Let $Y$ be a separable metric space such that
 every (non-singleton) connected subset of $Y$ has non-empty interior. 
Consider a random product $\varphi$  over a 
primitive Markov shift $(\Sigma_{k},\mathscr{F},\mathbb{P},\sigma)$
defined
on a compact and connected subset $M$ of $Y^{m}$. Suppose that $\varphi$ splits.
 Then
$\mathbb{P}^{-}(S_{\varphi})=1$.
\end{theorem}

This theorem is an important step of the proof of Theorem \ref{stable} 
and its proof is inspired by the ideas in \cite{Yuri}.

 \subsection{Proof of Theorem  \ref{abundance}}
 \label{ss.proofofabundance}
 
Fix $s$. Given $x\in \pi_{s}(M)$ define the set 
$$
\Sigma_{x}(s)\eqdef \{\xi\in \Sigma_{k}\colon x\in \pi_{s}(I_{\xi})\}.
$$
Consider also the set $S_{\varphi}(s)\eqdef\{\xi\in \Sigma_k \colon \pi_{s}(I_{\xi})\, \, \mbox{is a singleton}\}$.
Since $Y$ is separable there is a dense and countable subset $D$ of $Y$. 
Note that if
$\xi\not\in S_\varphi(s)$ then the set $\pi_{s}(I_\xi)$ is a connected subset of $Y$ which is not a singleton and hence, 
by hypothesis, its interior is not empty and thus
$\pi_{s}(I_\xi)$ contains
a point of $D$. This implies that
\begin{equation}\label{prin2}
( S_{\varphi}(s))^c=
\Sigma_{k}\setminus S_{\varphi}(s)\subset \bigcup_{x\in D\cap \pi_{s}(M)}\Sigma_{x}(s).
\end{equation}

\begin{proposition}\label{p.pzero}
$\mathbb{P}^{-}(\Sigma_x(s))=0$ for every $x\in  \pi_{s}(M)$.
\end{proposition}

In view of \eqref{prin2} this proposition implies  that $\mathbb{P}^{-}(S_{\varphi}(s))=1$. 
The theorem follows noting that 
$
S_{\varphi}=\bigcap_{s=1}^{n} S_{\varphi}(s).
$


\begin{proof}[Proof of Proposition~\ref{p.pzero}]
Fix $x\in \pi_{s}(M)$, recall the definition of $S_{n}^{x}(s)$ in \eqref{e.snx} and note
that for every
$n\geq 1$ it holds
$\Sigma_x(s)\subset S_{n}^{x}(s)$. Hence
$$
\Sigma_x(s) \subset \bigcap_{n\ge 1} S_{n}^x(s).
$$
Therefore, recalling that $S_{n+1}^{x}(s)\subset S_{n}^{x}(s)$, \eqref{e.inclusionSn}, it follows
\begin{equation}\label{SS}
\mathbb{P}^{-} (\Sigma_x(s)) \le 
\mathbb{P}^{-} \Big( \bigcap_{n\ge 1} S_{n}^x(s) \Big) =\lim_{n\to \infty} \mathbb{P}^{-} (S_n^x(s)).
\end{equation}
Hence to prove the proposition it is enough to see that 

\begin{lemma} \label{l.limite}
$\lim_{n\to \infty} \mathbb{P}^{-} (S_n^x(s))=0$. 
\end{lemma}
%
%
%

\begin{proof}
  By the splitting hypothesis there is a pair of $\mathbb{P}$-admissible cylinders $[a_{1}\dots a_{\ell}]$ and 
 $[b_{1}\dots b_{r}]$ with $a_{\ell}=b_{r}$ such that 
 
 \begin{equation}\label{e.splittingcondition}
\pi_{s}(f_{\omega}^{n}( f_{a_{\ell}}\circ\dots\circ f_{a_{1}}(M)))\cap \pi_{s}(f_{\omega}^{n}(f_{b_{r}}\circ\dots\circ f_{b_{1}}(M)))=\emptyset
\end{equation}
for every  $\pi_{s}$,  every $n\geq 0$, and every $\omega\in \Sigma_k$.

Next claim restates the splitting property adding the condition that the two sequences in that condition have the same length.

 \begin{claim} \label{c.claimera52}
 There are $\mathbb{P}^{-}$-admissible cylinders $[\xi_{0}\dots \xi_{N-1}]$
 and $[\eta_{0}\dots\eta_{N-1}]$
  with $\xi_{0}=\eta_{0}$, $\xi_{N-1}=\eta_{N-1}$, such that
\[
 \begin{split}
&\pi_{s}(f_{\omega}^{n}( f_{\xi_{0}}\circ\cdots\circ f_{\xi_{N-1}}(M)))\cap 
\pi_{s}(f_{\omega}^{n}(f_{\eta_{0}}\circ\cdots\circ f_{\eta_{N-1}}(M)))=\emptyset
\end{split}
\]
for every $s$, $n\geq 0$, and every $\omega$.
\end{claim}

\begin{proof}
Consider $a_1,\dots,a_\ell$ and $b_1,\dots,b_r$ as in \eqref{e.splittingcondition}.
 Since the transition matrix of  $\mathbb{P}$ is primitive then  
 there is $n_{0}$ such that
 for every $n\geq n_{0}$ there are $\mathbb{P}$-admissible cylinders of the form
  $[1 c_{n}\dots c_1 a_1]$ and $[1 d_{n}\dots d_1 b_1]$.
  Take now $n_{1},n_{2}\geq n_{0}$ 
 with $n_{1}+\ell=n_{2}+s$ and 
 consider $\mathbb{P}^{-}$-admissible cylinders
  $[1 c_{n_1}\dots c_1 a_1]$ and $[1 d_{n_2}\dots d_1 b_1]$.
  Let $N=n_{1}+\ell+2$ and observe that by construction (and since
 $[a_{1}\dots a_{\ell}]$ and 
 $[b_{1}\dots b_{r}]$ are both admissible)
   the cylinders 
   $$
  [\xi_{0}\dots \xi_{N-1}]=[a_{\ell}\dots a_{1}c_{1}\dots c_{n_{1}}1] \quad\mbox{and} \quad
[\eta_{0}\dots \eta_{N-1}]=[b_{r}\dots b_{1}d_{1}\dots d_{n_{2}}1]
$$  
are $\mathbb{P}^{-}$-admissible.
Note also that
$$
f_{c_{1}}\circ \dots \circ f_{c_{n_{1}}}\circ f_{1}(M)\subset M,
\qquad f_{d_{1}}\circ \dots \circ f_{d_{n_{2}}}\circ f_{1}(M)\subset M.
$$
Hence the splitting condition in \eqref{e.splittingcondition} for $a_{1}\dots a_{\ell}$ and 
 $b_{1}\dots b_{r}$  implies that
$\xi_{0}=\eta_{0}$, $\xi_{N-1}=\eta_{N-1}$, and  
$\xi_{0}\dots \xi_{N-1}$
 and $\eta_{0}\dots\eta_{N-1}$
satisfy the empty intersection condition in the claim.
\end{proof}

Let $[\xi_{0}\dots \xi_{N-1}]$ and  $[\eta_{0}\dots\eta_{N-1}]$ the cylinders given by 
Claim~\ref{c.claimera52}. Suppose that 
$0<\mathbb{P}^{-} ([\xi_{0}\dots \xi_{N-1}]) \le \mathbb{P}^{-} ([\eta_{0}\dots\eta_{N-1}]) $
and let  $W=[\xi_{0}\dots \xi_{N-1}]$.
We can now apply 
Proposition \ref{p.l.menorouigual} to
$W$ obtaining
\begin{equation}\label{e.lemma}
 \mathbb{P}^{-}(S_{\ell N}^x(s))\leq \mathbb{P}^{-}(\Sigma_{\ell}^W),
\quad \mbox{ for every $\ell\geq 1$.}  
\end{equation}

 We now estimate the right hand side  of \eqref{e.lemma}.  Let 
$$
\Sigma_{\infty}^{W}\eqdef \bigcap_{\ell\geq 1} \Sigma_{\ell}^{W}=
\{\omega\in \Sigma_{k} \colon \sigma^{iN}(\omega)
 \cap W=\emptyset\,\, \mbox{for all}\,\, i\geq 0\}. 
 $$
 
 \begin{remark}
 \label{r.primitive}
 {\em{
Since $P$ is primitive 
the shift
$(\Sigma_{k},\mathscr{F},\mathbb{P}^{-},\sigma)$ 
is mixing
and hence
the system
$(\Sigma_{k},\mathscr{F},\mathbb{P}^{-},\sigma^\ell)$ is ergodic for every $\ell\ge 1$,
see for instance \cite[page 64]{Mane}. }}
\end{remark}

Since $0<\mathbb{P}^{-} (W)$, by Remark~\ref{r.primitive} we can apply 
the Birkhoff's ergodic theorem to get that $\mathbb{P}^{-}(\Sigma_{\infty}^{W})=0$.
 Hence condition $\Sigma_{\ell+1 }^{W}\subset \Sigma_{\ell}^{W}$, 
 recall  \eqref{e.inclusioncylinders}, implies that 
 $$
 \lim_{\ell\to \infty} \mathbb{P}^{-}( \Sigma_{\ell}^{W})=0.
 $$
It follows from 
\eqref{e.lemma} that
$$
\lim_{\ell\to \infty} \mathbb{P}^{-} (S_{\ell N}^x(s))\leq \lim_{\ell\to \infty} \mathbb{P}^{-} (\Sigma_{\ell}^{W})=0.
$$ 
The lemma follows recalling again that $S_{n+1}^{x}(s)\subset S_{n}^{x}(s)$, see \eqref{e.inclusionSn}.
\end{proof}
The proof of Proposition \ref{p.pzero} is now complete
\end{proof}
The proof of Theorem \ref{abundance} is now complete. 
  \hfill \qed

\section{Stability of the Markov operator}
\label{s.Markovop}

Throughout this section $\varphi$ denotes a  random product  on a compact and connected 
metric space $(M,d)$ over a Markov shift $(\Sigma_{k},\mathscr{F},\mathbb{P},\sigma)$ as in \eqref{e.randomproductmap}.
In this section we will prove Theorem~\ref{stable} and Corollary  \ref{ergodictheorem}.
We begin by providing a handy form of the Markov operator of $\varphi$. 

\subsection{The Markov operator}
\label{ss.Markovop}
 Let $\widehat M\eqdef \{1,\dots,k\}\times M$. Given a subset $\widehat{B}\subset\widehat{M}$, its 
{\emph{$i$-section}} is defined by
$$
\widehat B_{i}\eqdef\{x\in M\colon (i,x)\in \widehat{B}\}. 
$$
The {\emph{$i$-section of a probability 
measure $\widehat{\mu}$}} on $\widehat{M}$ is the measure defined on  $M$ by
$$
\mu_{i}(B)\eqdef \widehat{\mu}(\{i\}\times B ), 
\quad \mbox{
where $B$ is any Borel subset of $M$.}
$$
Observe that 
${\mu}_i$ is a finite measure on $M$ but, in general, it is 
not a probability measure. 
Since the measure $\widehat \mu$ is completely defined by its sections
we write $\widehat{\mu}=(\mu_{1},\dots,\mu_{k})$
and note that   
$$
\widehat{\mu}(\widehat B)=\sum_{j=1}^{k}\mu_{j}(\widehat B_{j})
\quad \mbox{for every Borel subset $\widehat B$
of $\widehat{M}$}.
$$
Similarly,
given a function  $\widehat g\colon \widehat{M}\to \mathbb{R}$
we define its $i$-section
$g_i\colon M \to \mathbb{R}$
 by 
$g_{i}(x)\eqdef \widehat g(i,x)$ and write $\widehat g= \langle g_{1},\dots,g_{k} \rangle$.
By definition, it follows that
\begin{equation}\label{desintegracao}
\int \widehat g\, d\widehat\mu=\sum_{i=1}^{k} \int g_{i}\, d\mu_{i},
\quad 
\mbox{for every} 
\quad
\widehat{\mu}=(\mu_{1},\dots,\mu_{k})\in \mathcal{M}_{1}(\widehat{M}).
\end{equation}

For the next lemma recall that 
 $\phi_* \mu$ denotes the pushforward of the measure $\mu$ by $\phi$
(i.e., $\phi_{*} \mu (A)=  \mu (\phi^{-1} (A))$).

\begin{lemma} Consider a random product $\varphi(n,\omega,x)=f_{\omega}^{n}(x)$ on $M$ 
over a Markov shift $(\Sigma_{k},\mathscr{F},\mathbb{P},\sigma)$. Let $P=(p_{ij})$ be the transition
matrix 
of  $\mathbb{P}$. The Markov operator associated to $\varphi$ is given by 
$$
  T\widehat{\mu}(\widehat{B})\eqdef \sum_{i,j}p_{ij}f_{j*}\mu_{i}(\widehat B_{j}), 
$$
where $\widehat{\mu}=(\mu_{1},\dots,\mu_{k})\in \mathcal{M}_{1}(\widehat{M})$ and 
$\widehat  B$ is
any Borel subset of $\widehat M$. In particular,
$$
(T\widehat{\mu})_{j}=\sum_{i,j}p_{ij}f_{j*}\mu_{i}.
$$
\end{lemma}
\begin{proof}
Let $\widehat B$ be a Borel subset of $\widehat M$. The transition probability
on the set $\widehat M=\{1,\dots,k\}\times M $
associated to $\varphi$ is given by (recall \eqref{transition})
$$
\widehat P((i,z),\widehat{B})= \sum_{j=1}^{k}p_{ij}\mathds{1}_{\widehat{B}}(j,f_{j}(z)),
$$
and hence the corresponding Markov operator is given by (recall \eqref{handyform})
\[
\begin{split}
  T\widehat\mu(\widehat{B})&=\int \sum_{j=1}^{k}p_{ij}\mathds{1}_{\widehat{B}}(j,f_{j}(z))  \,d\widehat\mu(i,z)
  =\sum_{j=1}^{k}\int p_{ij}\mathds{1}_{\widehat{B}}(j,f_{j}(z))  \,d\widehat\mu(i,z)\\
& \underset{\textrm{by \eqref{desintegracao}}}=\sum_{j=1}^{k}\int p_{ij}\mathds{1}_{\widehat{B_{j}}}(f_{j}(z))  \,d\mu_{i}(z)\\
 &=\sum_{j=1}^{k}p_{ij}\int\mathds{1}_{\widehat{B_{j}}}(z)  \,df_{j*}\mu_{i}(z)
=\sum_{i,j}p_{ij}f_{j*}\mu_{i}(\widehat{B_{j}}),
\end{split}
\] 
proving the lemma. 
\end{proof}

\subsection{Shrinking of the reverse order iterates} 
\label{ss.shrink}
Recall the definition of the set 
 $S_{\varphi}$  
 in \eqref{weakly}
 and define  the
 {\emph{coding map}}\footnote{This is the standard terminology for the map $\pi$ 
 when
$S_{\varphi}=\Sigma_k$.}
\begin{equation}\label{coding}
\pi\colon S_\varphi \to M, 
\quad
\pi(\omega)\eqdef\lim_{n\to \infty}f_{\omega_{0}}\circ\cdots\circ f_{\omega_{n}}(p),
\end{equation}
where $p$ is any point of $M$.
By definition of the set $S_{\varphi}$, this limit always exists 
and does not depend on $p\in M$.

\begin{lemma}\label{interessante}
For every sequence $(\mu_{n})$
of probabilities of $\mathcal{M}_{1}(M)$  and every
 $\omega\in S_{\varphi}$ it holds
$$
\lim_{n\to \infty}f_{\omega_{0}*}\ldots_\ast f_{\omega_{n}*}\mu_{n}=\delta_{\pi(\omega)}.
$$
\end{lemma}

\begin{proof}
 Consider a sequence of probabilities $(\mu_n)$ and  $\omega\in S_{\varphi}$. Fix any 
$g\in C^{0}(M)$. Then  given any $\epsilon>0$ there is $\delta>0$ 
such that 
$$
|g(y)-g\circ\pi(\omega)|<\epsilon
\quad \mbox{for all $y\in M$ with $d(y,\pi(\omega))<\delta$.}
$$
Since $\omega \in S_{\varphi}$ there is $n_{0}$
such that  $d(f_{\omega_{0}}\circ\cdots \circ f_{\omega_{n}}(x),\pi(\omega))< \delta$ for every
$x\in M$ and every $n\geq n_{0}$. Therefore 
for $n\geq n_{0}$ we have
\[
\begin{split}
\left|
g\circ \pi(\omega)-
\int g\, df_{\omega_{0}*}\ldots_\ast f_{\omega_{n}*}\mu_{n}\right|&=
\left|
\int 
g\circ \pi(\omega)\,d\mu_{n}-
\int g\circ f_{\omega_{0}}\circ\cdots\circ f_{\omega_{n}}(x)\,d \mu_{n}
\right|\\
&\leq \int 
|g\circ \pi(\omega)- g\circ f_{\omega_{0}}\circ\cdots\circ f_{\omega_{n}}(x)|\, d\mu_{n}
\le \epsilon.
\end{split}
\]
This implies that
$$
\lim_{n\to \infty}\int g\, df_{\omega_{0}*}\ldots_\ast f_{\omega_{n}*}\mu_{n}=g\circ\pi(\omega)
$$
Since this holds for every continuous map $g$ the lemma follows.
\end{proof}

\subsection{Proof of Theorem \ref{stable}}
Let $\pi\colon S_{\varphi}\to M$ be the coding map in \eqref{coding} and
 define 
 $$
 \widehat\pi\colon   S_{\varphi}\to \widehat M, \quad 
\widehat \pi(\omega)\eqdef (\omega_{0},\pi(\omega)).
 $$
Since the random product $\varphi$ splits it follows from Theorem \ref{abundance}
 that $\mathbb{P}^{-}( S_{\varphi})=1$.
 Hence the map $\widehat \pi$ is defined $\mathbb{P}^{-}$-almost 
 everywhere in $\Sigma_{k}$. This allows us to consider the probability measure $\widehat \pi_{*}\mathbb{P}^{-}$.
 
 The next result is a reformulation of Theorem~\ref{stable}  (indeed, a stronger version of it)
 and implies that
 the probability measure $\widehat \pi_{*}\mathbb{P}^{-}$ is  the unique stationary measure
of the Markov operator and is attracting.

\begin{theorem}\label{t.p.attractingmarkov}
Let $\varphi$ be a random product on a compact metric space $M$ over a primitive Markov shift and suppose that 
$\mathbb{P}^{-}(S_{\varphi})=1$.
Given any  
measure
$\widehat{\mu}\in \mathcal{M}_{1}(\widehat{M})$
 and any continuous 
  $\widehat{g}\in C^{0}(\widehat{M} )$
it holds 
\[
\lim_{n\to \infty}\int \widehat{g}\, dT^{n}{\widehat{\mu}}=\int \widehat{g}
  \, d\widehat \pi_{*}\mathbb{P}^{-},
 \]
 where $T$ is the Markov operator of $\varphi$.
\end{theorem}

%
%

\begin{remark}\emph{
In \cite{Recurrent} it is proved that the condition $\mathbb{P}^{-}(S_{\varphi})=1$
implies the existence 
of a unique stationary measure and that $T^{n}\widehat{\mu}$ converge in the weak$*$-topology to the unique stationary measure provided that $\widehat{\mu}=(\mu_{1},\dots,\mu_{k})$ satisfy $\mu_{i}(M)=p_{i}$ for every
$i$, where $(p_{1},\dots,p_{k})$ is the stationary probability vector.
 The  result in \cite{Recurrent}  is a version of Letac principle \cite{Letac}
for a Markovian random product. Here we prove that $T^{n}\widehat{\mu}$ converges in the weak$*$-topology to the unique stationary measure  for every $\widehat{\mu}$ (and not only for measures 
$\widehat{\mu}=(\mu_{1},\dots,\mu_{k})$ with $\mu_{i}(M)=p_{i}$).}
\end{remark}

\begin{proof}[Proof of Theorem \ref{t.p.attractingmarkov}]
It follows from \eqref{desintegracao} that
 \begin{equation}\label{e.3101}
\int \widehat g\, dT^{n}{\widehat{\mu}}=\sum_{j=1}^{k}\int g_{j}
\,d  (T^{n}{\widehat{\mu}})_{j},
\quad 
T^{n}{\widehat{\mu}}=\big( (T^{n}{\widehat{\mu}})_{1}, \dots, (T^{n}{\widehat{\mu}})_{k}\big).
 \end{equation}


We prove 
the convergence of the integrals of the sum in  \eqref{e.3101}  
in three steps corresponding to Lemmas~\ref{l.primitiveee}, \ref{l.perronfrob}, and
\ref{l.quase} below.
 First,
 given a continuous function $g\colon M \to \mathbb{R}$
denote by $\Vert g\Vert $ its uniform
 norm.

\begin{lemma}\label{l.primitiveee}
Consider $\widehat{\mu}=(\mu_{1},\dots,\mu_{k})\in \mathcal{M}_{1}(\widehat{M})$ such that 
$\mu_{i}(M)>0$ for every $i\in \{1,\dots, k\}$. Then for every $g\in C^{0}(M)$ it holds
$$
\limsup_{n}\left|\int g \, d (T^{n}{\widehat{\mu}})_{j}
-\int_{[j]} g\circ\pi\, d\,\mathbb{P}^{-} \right|\leq k\, \Vert g\Vert\, \max_{i}|\mu_{i}(M)-p_i|,
$$
where $\bar p=(p_1,\dots,p_k)$ is the unique stationary vector of $P=(p_{ij})$ is the transition matrix of
$\mathbb{P}$. 
\end{lemma}

%

\begin{proof}
Take $\widehat{\mu}\in\mathcal{M}_{1}(\widehat{M})$ 
as in the lemma
and for each $i$ define
the
probability measure $\overline{\mu}_{i}$ 
$$
\overline{\mu}_{i}(B)\eqdef \frac{\mu_{i}(B)}{\mu_{i}(M)},
\quad \mbox{where $B$ is a Borel subset of $M$}.
$$
A straightforward calculation and the previous definition imply that  
\[
\begin{split}
 (T^{n}{\widehat{\mu}})_{j}&=\sum_{\xi_1,\dots ,\xi_{n}}
 p_{\xi_{n}\xi_{n-1}}\dots p_{\xi_{2}\xi_{1}}p_{\xi_{1}j} \, f_{j*}f_{\xi_{1}*}
 \dots_{*}f_{\xi_{n-1}*}\mu_{\xi_{n}}\\    
&=
 \sum_{\xi_1,\dots ,\xi_{n}}
 \mu_{\xi_{n}}(M)\, p_{\xi_{n}\xi_{n-1}}\dots p_{\xi_{2}\xi_{1}}p_{\xi_{1}j}
  f_{j*}f_{\xi_{1}*}\dots_{*}f_{\xi_{n-1}*}\overline{\mu}_{\xi_{n}}. 
 \end{split}
 \]
  Thus given any $g\in C^{0}(M)$ we have that
 \[ \int g\, d(T^{n}{\widehat{\mu}})_{j}= \sum_{\xi_1,\dots ,\xi_{n}}
 \mu_{\xi_{n}}(M)\, p_{\xi_{n}\xi_{n-1}}\dots 
 p_{\xi_{1}j}
  \int g \,d f_{j*}f_{\xi_{1}*}\dots_{*}f_{\xi_{n-1}*}\overline{\mu}_{\xi_{n}}.
\]
%
  
  Let 
  $$
  L_{n}\eqdef\left|\int g \, d (T^{n}{\widehat{\mu}})_{j}
-\int_{[j]} g\circ\pi\, d\,\mathbb{P}^{-} \right|
$$
 and write  $ \mu_{\xi_{n}}(M)=
 (\mu_{\xi_{n}}(M)-p_{\xi_{n}})+ p_{\xi_{n}}$. Then

\[
\begin{split}
L_{n}&\leq \left| \sum_{\xi_1,\dots ,\xi_{n}}
 (\mu_{\xi_{n}}(M)-p_{\xi_{n}})\, p_{\xi_{n}\xi_{n-1}}\dots 
 p_{\xi_{1}j}
  \int g \,d f_{j*}f_{\xi_{1}*}\dots_{*}f_{\xi_{n-1}*}\overline{\mu}_{\xi_{n}}\right|\\
  & + \left|\sum_{\xi_1,\dots ,\xi_{n}}
 p_{\xi_{n}}\, p_{\xi_{n}\xi_{n-1}}\dots 
 p_{\xi_{1}j}
  \int g \,d f_{j*}f_{\xi_{1}*}\dots_{*}f_{\xi_{n-1}*}\overline{\mu}_{\xi_{n}}-
  \int_{[j]} g\circ\pi\, d\,\mathbb{P}^{-}\right|\\
  &\leq \max_{i}|\mu_{i}(M)-p_i|
\Vert g\Vert\sum_{\xi_1,\dots ,\xi_{n}}
 p_{\xi_{n}\xi_{n-1}}\dots 
 p_{\xi_{1}j}\\
 & +
  \left|\sum_{\xi_1,\dots ,\xi_{n}}
 p_{\xi_{n}}\, p_{\xi_{n}\xi_{n-1}}\dots 
 p_{\xi_{1}j}
  \int g \,d f_{j*}f_{\xi_{1}*}\dots_{*}f_{\xi_{n-1}*}\overline{\mu}_{\xi_{n}}-
  \int_{[j]} g\circ\pi\, d\,\mathbb{P}^{-}\right|.\\
\end{split}
\]
To estimate the first term of this inequality note that $\sum_{\xi_1,\dots ,\xi_{n-1}}
 p_{\xi_{n}\xi_{n-1}}\dots
 p_{\xi_{1}j}$ is the entry $(\xi_{n},j)$ of the matrix $P^{n}$. Hence
 $$
 \sum_{\xi_1,\dots ,\xi_{n}}
 p_{\xi_{n}\xi_{n-1}}\dots 
 p_{\xi_{1}j}=\sum_{\xi_{n}=1}^k \, \sum_{\xi_1,\dots ,\xi_{n-1}}
 p_{\xi_{n}\xi_{n-1}}\dots
 p_{\xi_{1}j}\leq k.
  $$
  Therefore
  \begin{equation} \label{e.sumsum}
   \max_{i}|\mu_{i}(M)-p_i|
\Vert g\Vert\sum_{\xi_1,\dots ,\xi_{n}}
 p_{\xi_{n}\xi_{n-1}}\dots 
 p_{\xi_{1}j} \le 
 k\, \Vert g\Vert\, \max_{i}|\mu_{i}(M)-p_i|.
  \end{equation}
 
 We now estimate the second term in the sum above. 
 \begin{claim}\label{c.segunda25} For every continuous function $g$ it holds
 $$
 \lim_{n\to \infty}\sum_{\xi_1,\dots ,\xi_{n}}
 p_{\xi_{n}}p_{\xi_{n}\xi_{n-1}}\dots 
 p_{\xi_{1}j}
  \int g \,d f_{j*}f_{\xi_{1}*}\dots_{*}f_{\xi_{n-1}*}\overline{\mu}_{\xi_{n}}
  =\int_{[j]} g\circ\pi\, d\,\mathbb{P}^{-}.$$
 \end{claim}
 
Observe that equation \eqref{e.sumsum} and the claim imply the
 lemma.

 \begin{proof}[Proof of Claim \ref{c.segunda25}]
 Consider the sequence of functions 
given by
$$
G_{n}:\Sigma_{k}^+\rightarrow \mathbb{R}, \quad 
G_{n}(\xi) \eqdef\displaystyle\int g\,d f_{\xi_{0}*}f_{\xi_{1}*}\dots_{*}f_{\xi_{n-1}*}\overline{\mu}_{\xi_{n}}.
$$
By definition,
for every $n$ the 
 map
$G_{n}$ is constant in the cylinders $[\xi_{0},\ldots,\xi_{n}]$ and thus
it is  measurable. By definition of $\mathbb{P}^{-}$, for every $j$ we have that 
$$
p_{\xi_{n}}p_{\xi_{n}\xi_{n-1}}\dots p_{\xi_{2}\xi_{1}}p_{\xi_{1}j}
=\mathbb{P}([\xi_{n}\xi_{n-1}\dots\xi_{1}j])=\mathbb{P}^{-}([j\xi_{1}\xi_{2}\dots\xi_{n}]).
$$
Hence 
$$
 \sum_{\xi_1,\dots ,\xi_{n}}
 p_{\xi_{n}}p_{\xi_{n}\xi_{n-1}}\dots 
 p_{\xi_{1}j}
  \int g \,d f_{j*}f_{\xi_{1}*}\dots_{*}f_{\xi_{n-1}*}\overline{\mu}_{\xi_{n}}=\int_{[j]} G_{n}\,d\,\mathbb{P}^{-}.
$$
It follows from  Lemma \ref{interessante} that
\begin{equation}\label{e.limitFn2}
\lim_{n\rightarrow\infty} G_{n}(\xi)=g\circ \pi(\xi)
\quad 
\mbox{for  $\mathbb{P}^{-}$-almost every  $\xi$}.
\end{equation}

  Now note that $|G_{n} (\xi)|\leq \|g\|$ for every $\xi\in \Sigma_{k}^{+}$. 
From \eqref{e.limitFn2}, using the dominated convergence theorem, 
we get
$$
\lim_{n\rightarrow\infty}\int_{[j]} G_{n}\, d\,\mathbb{P}^{-}
=\int_{[j]} g\circ\pi\, d\,\mathbb{P}^{-},
$$
proving the claim.
 \end{proof}
%
%
The proof of the lemma is now complete.
 \end{proof}
 
 \begin{remark}{\em{
 Recall that, by hypothesis,  the transition matrix $P$ is primitive. Hence 
by the Perron-Frobenius theorem (see for instance \cite[page 64]{Mane})
there is a unique positive stationary probability vector  $\bar p=(p_1,\dots, p_k)$  of $P$
such that for every probability vector $\widehat{p}$ we have
\begin{equation}\label{e.perronfrob}
\lim_{n\to \infty}\widehat{p}\,P^{n}=\bar{p}
\quad
\mbox{for every probability vector $\widehat{p}$}.
\end{equation}
The vector $\bar p$ is the stationary vector of $P$.}}
\end{remark}
 
\begin{lemma}\label{l.perronfrob}
Let $\widehat{\mu}=(\mu_1,\dots, \mu_k)\in \mathcal{M}_{1}(\widehat{M})$. 
Then
 $$
 \lim_{n\to \infty}\max_{i}
|(T^{n}\widehat{\mu})_{i}(M)-p_i| = 0.
$$  
In particular, there is  
$n_{0}$ such that 
the vector $\big((T^{n}\widehat{\mu})_{1}(M),\dots,(T^{n}\widehat{\mu})_{k}(M)\big)$
 is positive for every $n\geq n_{0}$.
\end{lemma}

 \begin{proof}
Note that by definition of the Markov operator
  $$
  \big( (T\widehat{\mu})_{1}(M),\dots,(T\widehat{\mu})_{k}(M)\big)=
 \widehat{p}\,P, \quad
 \mbox{ where $\widehat{p}=(\mu_{1}(M),\dots,\mu_{k}(M))$}.
 $$
  Hence for every $n\geq 1$
 \begin{equation}\label{e.perron} 
( T^n \widehat \mu(M))=
  \big( (T^{n}\widehat{\mu})_{1}(M),\dots,(T^{n}\widehat{\mu})_{k}(M) \big)=
 \widehat{p}\,P^{n}.
  \end{equation}
  Now the lemma follows from \eqref{e.perronfrob}.
  \end{proof}

\begin{lemma}\label{l.quase}
Let $\widehat{\mu}=(\mu_1,\dots, \mu_k)\in \mathcal{M}_{1}(\widehat{M})$. Then 
for every function $g\in C^{0}(M)$ it holds
 \[
 \lim _{n\to \infty}\int g \, d (T^{n}{\widehat{\mu}})_{j}
=\int_{[j]} g\circ\pi\, d\,\mathbb{P}^{-}.
 \]
\end{lemma}

\begin{proof}
By Lemma~\ref{l.perronfrob}, we can apply
 Lemma \ref{l.primitiveee} to the measure $T^{n_{1}}({\widehat{\mu}})$ 
 for every $n_{1}\geq n_{0}$, 
 obtaining  the following inequality for every $g\in C^{0}(M)$,
$$
\limsup_{n}\left|\int g \, d (T^{n+n_{1}}{\widehat{\mu}})_{j}
-\int_{[j]} g\circ\pi\, d\,\mathbb{P}^{-} \right|\leq k\, \Vert g\Vert\, \max_{i}
|(T^{n_{1}}\widehat{\mu})_{i}(M)-p_i|.
$$
 It follows from the definition of $\limsup$ and  
 the previous inequality  that 
$$
 \limsup_{n}\left|\int g \, d (T^{n}{\widehat{\mu}})_{j}
-\int_{[j]} g\circ\pi\, d\,\mathbb{P}^{-} \right|\leq k
\, \Vert g\Vert \,
\max_{i}
|(T^{n_{1}}\widehat{\mu})_{i}(M)-p_i|
 $$
 for every $n_{1}\geq n_{0}$. 
 The lemma now follows from Lemma~\ref{l.perronfrob}.
 \end{proof}

 To get the limit in the proposition 
  take  $\widehat{g}=\langle g_{1},\dots,g_{k} \rangle$,
 apply  Lemma \ref{l.quase}
 to the maps $g_{i}$,  and use \eqref{e.3101}
 to get
 $$
 \lim_{n\to \infty}\int \widehat g\, d \,T^{n}{\widehat{\nu}}
 \underset{\tiny{\mbox{\eqref{e.3101}}}}{=}
\sum_{j=1}^{k}\lim_{n\to \infty}  \int g_{j}
\,d \, (T^{n}{\widehat{\nu}})_{j}
 \underset{\tiny{\mbox{L. \ref{l.quase}}}}{=}
\sum_{j=1}^{k}\int_{[j]} g_{j}\circ\pi\,d\,\mathbb{P}^{-}.
$$
Now observing that $g_{j}\circ\pi(\xi)=\widehat g\circ \widehat \pi (\xi)$ for every $\xi\in[j]$,
we conclude that 
$$
 \lim_{n\to \infty}\int \widehat g\, d \,T^{n}({\widehat{\nu}})
 =\sum_{j=1}^{k}\int_{[j]} \widehat g\circ \widehat \pi \, d\, \mathbb{P}^{-}=\int \widehat g\, 
 d\widehat \pi_{*}\mathbb{P}^{-},
 $$
 ending the proof of Theorem~\ref{t.p.attractingmarkov}.
\end{proof}

\subsection{Proof of Corollary \ref{ergodictheorem} }
We first see that
 the coding map defined in 
\eqref{coding} is
the unique invariant map  of the random product (recall the definition in 
 \eqref{f1omega}). 
 
 Observe that by the definition of $\pi$ we have
 \begin{equation}\label{inv}
 f_{\omega_{0}}(\pi(\sigma(\omega)))= \pi(\omega), \quad \mbox{for every $\omega\in S_{\varphi}$}.
 \end{equation}
By Theorem \ref{abundance} we have that \eqref{inv} holds $\mathbb{P}^{-}$-almost everywhere, 
getting the invariance of the coding map. 
Let us now prove that $\pi$ is the unique invariant map.
Let $\rho$ be an invariant map. Inductively we get
\begin{equation}\label{iteinv}
f_{\omega_{0}}\circ \dots \circ f_{\omega_{n}}(\rho(\sigma^{n}(\omega))=\rho(\omega) 
\end{equation}
for $\mathbb{P}^{-}$-almost everywhere. Thus, 
it is enough to consider $\omega\in S_\varphi$ satisfying \eqref{iteinv},
obtaining that 
$$
\rho(\omega)=\lim_{n\to \infty} f_{\omega_{0}}\circ \dots \circ f_{\omega_{n}}(\rho(\sigma^{n}(\omega))=\pi(\omega),
$$
 for $\mathbb{P}^{-}$-almost everywhere. Note that the existence of the limit follows from $\omega \in S_\varphi$.

We now prove the convergence result in the corollary. It follows from the proof of Theorem \ref{stable}
that for every $x\in M$ the Markov chain $ Z_{n}^{x}(\omega)=(\omega_{n-1}, X_{n}^{x}(\omega))$ has a unique stationary measure given by $\widehat{\pi}_{*}\mathbb{P}^{-}$, where $\widehat{\pi}(\omega)=(\omega_{0},\pi(\omega))$.  This allows us to apply  the Breiman ergodic theorem, see \cite{Breiman},
to obtain that for $\mathbb{P}$-almost every $\omega$ it holds
$$
 \lim_{n\to \infty}\frac{1}{n}\sum_{j=0}^{n-1}\widehat\phi( Z_{j}^{x}(\omega))=\int \widehat\phi(i,x)\,d\widehat\pi_{*}\mathbb{P}^{-}(i,x),
 $$
for every continuous function $\widehat\phi\colon \widehat X\to \mathbb{R}$. 
Hence
for $\mathbb{P}$-almost every $\omega$ we have that 
\begin{equation}\label{ergodic2}
 \lim_{n\to \infty}\frac{1}{n}\sum_{j=0}^{n-1}\widehat\phi( Z_{j}^{x}(\omega))=\int \widehat\phi(\widehat\pi(\omega))\,d\mathbb{P}^{-}(\omega)=\int \widehat\phi((\omega_{0},\pi(\omega))\,d\mathbb{P}^{-}(\omega).
\end{equation}

Now take a continuous function $\phi\colon M\to \mathbb{R}$. If we see $\phi$ as a function on $\widehat{M}$ that 
does not depend on its first coordinate then it follows from \eqref{ergodic2} that for $\mathbb{P}$-almost every $\omega$ 
we have 
$$
\lim_{n\to \infty}\frac{1}{n}\sum_{i=0}^{n-1}\phi(X_{i}^{x}(\omega))=
\int \phi (\pi(\omega)) d\mathbb{P}^- (\omega)=
\int \phi\,d(\pi_{*}\mathbb{P}^{-}),
$$
completing the proof of the corollary.
 \hfill \qed
 
\section{Synchronization of Markovian random products}
\label{s.sincro}
In this section, we prove Theorems \ref{presincronizacao} and~\ref{t.mcsincronizacao},
see Sections~\ref{s.presincronizacao}  and \ref{s.mcsincronizacao}.
In what follows, let $Y$ be metric space and 
consider a random product $\varphi(n,\omega,x)=f^{n}_{\omega}(x)$ defined on 
a compact subset  $M$ of $Y^{m}$ over an irreducible Markov
shift $(\Sigma_{k},\mathscr{F},\mathbb{P}, \sigma)$ and  suppose that $\varphi$ splits.

\subsection{An auxiliary lemma}
For every $n\geq 0$ and every $s$ we define two sequences of  random sets
$$
J^{s}_{n}(\omega)\eqdef \pi_{s}(f_{\omega_{0}}\circ \dots \circ f_{\omega_{n-1}}(M))\quad\mbox{and}\quad
 I^{s}_{n}(\omega)\eqdef
  \pi_{s}(f_{\omega_{n-1}}\circ\dots \circ f_{\omega_{0}}(M))=\pi_{s}(f_{\omega}^{n}(M)).
$$
It follows from the definition of the inverse Markov measure that 
$$
\mathbb{P}(x\in I^{s}_{n})\eqdef\mathbb{P}(\{\omega\in\Sigma_{k} \colon x\in I^{s}_{n}(\omega)\})=\mathbb{P}^{-}(\{\omega\in \Sigma_{k}\colon x\in J^{s}_{ n}(\omega)\})\eqdef \mathbb{P}^{-}(x\in J^{s}_{n}).
$$  

We begin with the following auxiliary  lemma:

\begin{lemma}\label{fastfast}
There is $N\geq 1$ and $0<\lambda<1$ such that for every $s$ we have 
$$
\mathbb{P}(x\in I^{s}_{n})\leq \lambda^{n},
\quad 
\mbox{for every $n\geq N$}
$$ 
\end{lemma}
\begin{proof}

We need the following claim that has the same flavor as  Claim~\ref{c.claimera52}.

\begin{claim}\label{c.flavor}
There are $\mathbb{P}^{-}$ admissible cylinders 
  $[\xi_{0}\dots \xi_{N-1}]$ and $[\eta_{0}\dots \eta_{N-1}]$ such that
$\xi_{0}=\eta_{0}$, $\xi_{N-1}=\eta_{N-1}$, and $p_{\xi_0 j}>0$ for all $j$, 
satisfying 
$$
 \pi_{s}(f_{\omega}^{n}( f_{\xi_{0}}\circ\cdots\circ f_{\xi_{N-1}}(M)))\cap 
\pi_{s}(f_{\omega}^{n}(f_{\eta_{0}}\circ\cdots\circ f_{\eta_{N-1}}(M)))=\emptyset
$$
for every $n\geq 0$, every $s$, and every $\omega$.
\end{claim}
\begin{proof}
By the splitting hypothesis, there is a pair of $\mathbb{P}$-admissible cylinders $[a_{1}\dots a_{\ell}]$ and 
 $[b_{1}\dots b_{r}]$ with $a_{\ell}=b_{r}$ such that for every $u$ it holds
 \begin{equation}\label{splitsplit}
\pi_{s}(f_{\omega}^{n}( f_{a_{\ell}}\circ\dots\circ f_{a_{1}}(M)))\cap \pi_{s}(f_{\omega}^{n}(f_{b_{r}}\circ\dots\circ f_{b_{1}}(M)))=\emptyset
 \end{equation}
for every $n\geq 0$, every $s$, and every $\omega$.
By hypothesis, there is $u$ such that $p_{uj}>0$ for every $j$, thus we can assume that $r=\ell$ and $b_{1}=a_{1}=u$. Since the transition matrix $P$ is irreducible there is 
a finite sequence $c_{1}\dots c_{\ell_{0}}$ such that the cylinder $[uc_{1}\dots c_{\ell_{0}}a_{\ell}]$ is $\mathbb{P}^{-}$-
admissible. Consider the two $\mathbb{P}^{-}$- admissible cylinders defined by
  $$
  [\xi_{0}\dots \xi_{N-1}]\eqdef[uc_{1}\dots c_{\ell_{0}} a_{\ell}\dots a_{1}] \quad\mbox{and} \quad
[\eta_{0}\dots \eta_{N-1}]\eqdef[uc_{1}\dots c_{\ell_{0}} b_{r}\dots b_{1}].
$$  
By construction, $u=\xi_{0}=\eta_{0}$ and $\xi_{N-1}=\eta_{N-1}$. It follows from the splitting hypothesis 
\eqref{splitsplit}
that 
$$
 \pi_{s}(f_{\omega}^{n}(f_{\xi_{0}}\circ\cdots\circ f_{\xi_{N-1}}(M)))\cap 
\pi_{s}(f_{\omega}^{n}(f_{\eta_{0}}\circ\cdots\circ f_{\eta_{N-1}}(M)))=\emptyset,
$$
proving the claim.
\end{proof}

Consider the two
$\mathbb{P}^{-}$ admissible cylinders 
  $[\xi_{0}\dots \xi_{N-1}]$ and $[\eta_{0}\dots \eta_{N-1}]$ given by the claim.
Assume that $0<\mathbb{P}^-( [\xi_{0}\dots \xi_{N-1}])\le \mathbb{P}^-( [\eta_{0}\dots \eta_{N-1}])$.
It follows from Proposition~\ref{p.l.menorouigual} that the cylinder $W=
[\xi_{0}\dots \xi_{N-1}]$ satisfies 
\begin{equation}\label{uni}
\mathbb{P}^{-}( x\in J_{\ell N}^{s})\leq \mathbb{P}^{-}(\Sigma_{\ell}^{W}), 
\quad 
\mbox{for every $\ell\geq 1$ and every $s$,}
\end{equation}
 where $\Sigma_{\ell}^{W}$ is the set defined in \eqref{particao}.
Next claim states that $(\mathbb{P}^{-}(\Sigma_{\ell}^{W}))_\ell$ converges exponentially fast to $0$
as $\ell\to \infty$.
\begin{claim}\label{fast}
There is $\lambda_{0}<1$ such that
$\mathbb{P}^{-}(\Sigma_{\ell}^{W})\leq \lambda_{0}^{\ell}$
 for every $\ell\geq 1$.
 \end{claim}
\begin{proof}
Let $Q=(q_{ij})$ and $\bar p=(p_{1},\dots,p_{k})$ be the transition matrix and the stationary 
probability vector, respectively, that determine the measure $\mathbb{P}^{-}$.
 Let 
$$
\rho\eqdef\inf_{j} q_{j\xi_{0}}q_{\xi_{0}\xi_{1}}\dots q_{\xi_{N-2}\xi_{N-1}}
\quad \mbox{and}\quad
\rho_{0}\eqdef\min\{\rho,\mathbb{P}^{-}(W)\}.
$$
Note that since 
  $p_{\xi_0 j}>0$ (and hence $q_{j \xi_0}>0$, recall \eqref{e.Q}) for all $j$ we have that $\rho>0$ 
and hence $0<\rho_{0}<1$. 
Note that 
$$
\mathbb{P}^{-}(\Sigma_{1}^{W})=1-\mathbb{P}^{-}(W)\leq 1-\rho_{0}.
$$ 
 Suppose, by induction, that $\mathbb{P}^{-}(\Sigma_{n}^{W})\leq (1-\rho_{0})^{n}$ for every 
 $1\le n\le \ell$.

%
%

Given   $C=[a_{0}\ldots a_{\ell N-1}]$ consider the concatenated
cylinder $C\ast [\xi_{0}\dots \xi_{N-1}]=[a_{0}\ldots a_{\ell N-1}\xi_{0}\dots \xi_{N-1}]$. 
Note that by  the definitions of $C$ and $\rho_0$ we have
$$
\mathbb{P}^- (C\ast [\xi_{0}\dots \xi_{N-1}] )\ge \rho_0 \mathbb{P}^- (C).
$$
Recall  the definition of the cylinders $E_\ell$ in \eqref{e.Eell}. By definition 
and \eqref{e.twoidentities},
 $$
 \Sigma_{\ell +1}^{W}=\bigcup_{C\in E^{\ell}}(C -C\ast [\xi_{0}\dots \xi_{N-1}]).
 $$ 
Since the above union is disjoint, we have that
\[
 \begin{split}
 \mathbb{P}^{-}(\Sigma_{\ell +1}^{W})&=\sum_{{C\in E^{\ell}}}\
 \mathbb{P}^{-}(C)-\mathbb{P}^{-}(C\ast [\xi_{0}\dots \xi_{N-1}])\\
 &\leq  \sum_{{C\in E^{\ell}}}\
 \mathbb{P}^{-}(C)-\mathbb{P}^{-}(C)\rho_{0}\\
 &=(1-\rho_{0})\sum_{{C\in E^{\ell}}}
 \mathbb{P}^{-}(C)\leq (1-\rho_{0})^{\ell +1}.
 \end{split}
 \]
 The claim now follows taking $\lambda_{0}=1-\rho_{0}\in (0,1)$. 
%
 \end{proof}

 We are now ready to conclude the proof of the lemma. 
 Let $\lambda_{1}=\lambda_{0}^{\frac{1}{N}}$, where  $\lambda_{0}$ is as in Claim \ref{fast}.
 From \eqref{uni} and Claim~\ref{fast} it follows that
 $$
 \mathbb{P}^{-}( x\in J_{\ell N}^{s})\leq \lambda_{1}^{\ell N}, \quad
 \mbox{for every $\ell\geq 1$.}
 $$
 Now take $\lambda<1$ with $\lambda_{1}\leq \lambda^{N}<\lambda$.
 We claim that 
$$
\mathbb{P}^{-}( x\in I_{ n}^{s})\leq \lambda^{n}, \quad 
\mbox{for every 
$n\geq N$.} 
$$
To see why this is so given any $n\geq N$ write $n=\ell N+r$ for some integer $\ell \geq 1 $ 
and $r\in \{0,\dots ,N-1\}$. Thus, observing that $J_{ n+1}^{s}(\omega) \subset J_n^{s}(\omega)$,  we have 
$$
\mathbb{P} ( x\in I_{ n}^{s})= \mathbb{P}^{-}( x\in J_{ n}^{s})\leq
\mathbb{P}^{-}(x\in J_{\ell N}^{s}) \leq \lambda_{1}^{\ell N}\leq 
\lambda^{\ell N-1}\lambda^{r+1}=\lambda^{n},
$$
proving  the lemma.
\end{proof}

\subsection{Proof of Theorem~\ref{presincronizacao}}
\label{s.presincronizacao}
Let $\mu$ be a probability measure defined on $Y$.
To prove the theorem we need to see that there is $q<1$ such that
for $\mathbb{P}$-almost every $\omega$ there is constant 
$C=C(\omega)$ such that
 $$
\mu (I^{s}_{n}(\omega))\leq C q^{n}, \quad
\mbox{for every $n\geq 1$}.
$$
Applying Fubini's theorem and Lemma \ref{fastfast} we get 
$$
\mathbb{E}(\mu( (I_{n}^{s})))\eqdef\int \mu(I_{n}^{s}(\omega))\, d\mathbb{P}=
\int\mathbb{P}( x\in I_{n}^{s})\, d\mu \leq \lambda^{n},
$$
for every $n\geq N$, where $N$ and $\lambda$ are as in Lemma \ref{fastfast}.

In particular, taking  any $q<1$ with $\lambda<q$ and applying the 
Monotone Convergence Theorem we have that 
$$
\displaystyle\int\sum _{n=1}^{\infty}\frac{ \mu (I_{n}^{s}(\omega))}{q^{n}}\, d\mathbb{P}=
\displaystyle\sum _{n=1}^{\infty}\frac{\mathbb{E}(\mu (I_{n}^{s}))}{q^{n}}<\infty.
$$ 
It follows that 
$$
\sum _{n=1}^{\infty}\frac{ \mu(I_{n}^{s}(\omega))}{q^{n}}< \infty,
$$
 for 
$\mathbb{P}$-almost every $\omega$.
Therefore for $\mathbb{P}$-almost every $\omega$ there 
is $C_{s}=C_{s}(\omega)$ such that 
$$
\mu(I_{n}^{s}(\omega))\leq C q^{n}.
$$ 
 To complete the proof of the theorem note that 
 $\mu ( \pi_{s}(\,f^{n}_{\omega}(M)))=\mu(I^{s}_{n}(\omega))$.
 \hfill \qed

\subsection{Proof of Theorem~\ref{t.mcsincronizacao}}
\label{s.mcsincronizacao}
To prove the theorem  observe that for any sub-interval $J$ of  $[0,1]$ we have that $\text{diam}\, J=m(J)$, where $m$ is the Lebesgue measure on $[0,1]$. Recall that we are considering 
diameter with respect to the metric $d(x,y)=\sum_{i}|x_{i}-y_{i}| $. Hence for every $X\subset \mathbb{R}^{m}$ we have that 
\begin{equation}\label{diam}
\mbox{diam} (X) \leq \sum_{s=1}^{m}\mbox{diam}\,(\pi_{s}(X)).
\end{equation}

Now it  is enough to apply Theorem~\ref{presincronizacao} to the Lebesgue measure on $\mathbb{R}$. Then 
 for every $s$ there is a set $\Omega_{s}$ with $\mathbb{P}(\Omega_{s})=1$ such that for every  $\omega\in\Omega_{s}$ there is  $C_{s}(\omega)>0$ such that
\begin{equation}\label{ssss}
\mbox{diam} ( \,\pi_{s}(f^{n}_{\omega}(M)))\leq C_{s}(\omega) q^{n}, \quad \mbox{for every} \quad n\geq 1,
\end{equation}
for some constant $0<q<1$ independent of $s$.
Let   $\Omega_0 \eqdef \bigcap \Omega_{s}$.
 Given  $\omega\in \Omega_0$ define
 $C(\omega)\eqdef\max_{s}C_{s}(\omega)$. Then it follows from \eqref{diam} and \eqref{ssss} that 
$$
\mbox{diam}(f^{n}_{\omega}(M))\leq C(\omega)q^{n},
\quad
\mbox{for every $\omega\in \Omega_{0}$}. 
$$
Since $\mathbb{P}(\Omega_{0})=1$ the theorem follows. 
\hfill \qed

\section{The splitting condition}
\label{ss.splittingcondition}
In this section, we prove Theorem \ref{specialclass}.
Given $(t_{1},\dots,t_{m})\in \{+,-\}^{m}$
consider the subset $A=A(t_{1},\dots,t_{m})$
of $\mathbb{R}^m\times \mathbb{R}^m$  
defined as follows. A point $(x,y)\in \mathbb{R}^m\times \mathbb{R}^m$ belongs to
$A(t_{1},\dots,t_{m})$ if and only if
\begin{itemize}
\item 
$\pi_s(x)< \pi_s(y)$ if $t_{s}=+$ and
\item  
$\pi_s(x)> \pi_s(y)$ if $t_{s}=-$.
\end{itemize}
Recall also the definition of 
$S_{M}(t_{1},\dots t_{m})$ in Section~\ref{sss.monotone}.

\begin{claim}
\label{tclaim}
Let $f\in S_{M}(t_{1},\dots t_{m})$.
\begin{itemize}
\item
If $t_{1}=+$ then for every $n\geq 0$ it holds
$$
(x,y)\in A \implies (f_{\omega}^{n}(x),f_{\omega}^{n}(y))\in A
$$
\item
If $t_{1}=-$ then for every $n\geq0$ it holds
$$
(x,y)\in A \implies (f_{\omega}^{2n}(x),f_{\omega}^{2n}(y))\in A\quad \mbox{and}\quad
(f_{\omega}^{2n+1}(y),f_{\omega}^{2n+1}(x))\in A.
$$
\end{itemize}
\end{claim}

\begin{proof}
For the first item just note that if $f\in S_{M}(t_{1},\dots,t_{m})$ then
 $(x,y)\in A$ implies that $(f(x),f(y))\in A$.

For the second item it is sufficient to note that if $f\in S_{M}(t_{1},\dots, t_{m})$ then
 $(x,y)\in A$ implies that $(f(y),f(x))\in A$.
\end{proof}

To prove the theorem first observe that if a
subset $I\times J$  of $\mathbb{R}^m \times \mathbb{R}^m$ is contained in $A(t_{1},\dots,t_{m})$ then  
\begin{equation}\label{decom}
\pi_{s}(I)\cap \pi_{s}(J)=\emptyset, \quad \mbox{for every $\pi_{s}$}. 
\end{equation}

Recall that the sets $M_{1}=
f_{a_{\ell}}\circ\dots\circ f_{a_{1}}(M)$ and $M_{2}=f_{b_{r}}\circ\dots\circ f_{b_{1}}(M)$
satisfy 
\begin{itemize}
\item 
$\pi_s(M_1)< \pi_s(M_2)$ if $t_{s}=+$ and
\item  
$\pi_s(M_1)> \pi_s(M_2)$ if $t_{s}=-$.
\end{itemize}
Hence $M_1\times M_2\subset A$.

Suppose first that $t_{1}=+$.
It follows from the first item in Claim \ref{tclaim} that for every $n\geq 0$ we have $f_{\omega}^{n}(M_{1})\times f_{\omega}^{n}(M_{2})\subset A$. Applying  
\eqref{decom} we get
that 
$$
\pi_{s}(f_{\omega}^{n}(M_{1}))\cap \pi_{s}(f_{\omega}^{n}(M_{2}))= \emptyset, 
\quad
\mbox{for every $\pi_{s}$.}
$$

We now consider the case $t_{1}=-$. It follows from the second item in Claim \ref{tclaim} that for every $n\geq 0$ we have that 
 either $f_{\omega}^{n}(M_{1})\times f_{\omega}^{n}(M_{2})\subset A$ or $f_{\omega}^{n}(M_{2})\times f_{\omega}^{n}(M_{1})\subset A$. Therefore, again from \eqref{decom}, it follows  that for every $n\geq 0$ we have that 
 $$
\pi_{s}(f_{\omega}^{n}(M_{1}))\cap \pi_{s}(f_{\omega}^{n}(M_{2}))= \emptyset,
\quad
\mbox{for every $\pi_{s}$,}
$$
ending the proof of the theorem.
\hfill \qed

\bibliographystyle{acm}
\bibliography{references}

\end{document}